\documentclass[11pt,final]{article}
\usepackage{amsthm,amsfonts,graphicx,hyperref,tikz,enumerate,psfrag}
\usepackage{fullpage}

\usepackage[utf8]{inputenc} 
\RequirePackage{amsthm,amsmath,amsfonts,tikz}  

\newtheorem{definition}{Definition}
\newtheorem{lemma}{Lemma}
\newtheorem{remark}{Remark}
\newtheorem{proposition}[lemma]{Proposition}
\newtheorem{theorem}[lemma]{Theorem}

\newcommand{\EE}{{\mathbf{E}}}
\newcommand{\VV}{{\mathbf{Var}}}
\newcommand{\PP}{{\mathbf{P}}}
\newcommand{\RR}{{\mathbb{R}}}

\newcommand{\D}{\,\rm{d}}

\newcommand{\dN}{\mathbb {N}}
\newcommand{\dR}{\mathbb {R}}

\newcommand{\dL}{\mathbb {L}}
\newcommand{\cW}{\mathcal {W}}
\newcommand{\cF}{\mathcal {F}}
\newcommand{\cB}{\mathcal {B}}
\newcommand{\cT}{\mathcal {T}}
\newcommand{\cQ}{\mathcal {Q}}
\newcommand{\cI}{\mathcal {I}}
\newcommand{\cJ}{\mathcal {J}}

\newcommand{\cO}{\mathcal {O}}

\newcommand{\cE}{\mathcal {E}}
\newcommand{\cG}{\mathcal {G}}

\newcommand{\cL}{\mathcal {L}}
\newcommand{\cP}{\mathcal {P}}
\newcommand{\cD}{\mathcal {D}}
\newcommand{\w}{\mathbf{w}}
\newcommand{\wm}{{w_{\rm{min}}}}
\newcommand{\wpi}{{\pi_{h}^{-}}}

\newcommand{\h}{\mathbf{h}}

\newcommand{\wcD}{\widetilde{\cD}}

\newcommand{\veps}{\varepsilon}

\definecolor{darkblue}{rgb}{0,0.3,0.9}

\title{Random walk on sparse random digraphs}
\author{Charles Bordenave, Pietro Caputo, Justin Salez}

\date{September 27, 2015; revised April 27, 2017}

\begin{document}
\maketitle

\begin{abstract}
A finite ergodic Markov chain exhibits \emph{cutoff} if its distance to equilibrium remains close to its initial value over a certain number of iterations and then abruptly drops to near 0 on a much shorter time scale. Originally discovered in the context of card shuffling (Aldous-Diaconis, 1986), this remarkable phenomenon is now rigorously established for many reversible chains. Here we consider the non-reversible case of random walks on sparse directed graphs, for which even the equilibrium measure is far from being understood. We work under the configuration model, allowing both the in-degrees and the out-degrees to be freely specified. We establish the cutoff phenomenon, determine its precise window and prove that the cutoff profile approaches a universal shape. We also provide a detailed description of the equilibrium measure. 
\end{abstract}

\begin{figure}[h!]
\begin{center}
\includegraphics[angle =0,height = 6.5cm]{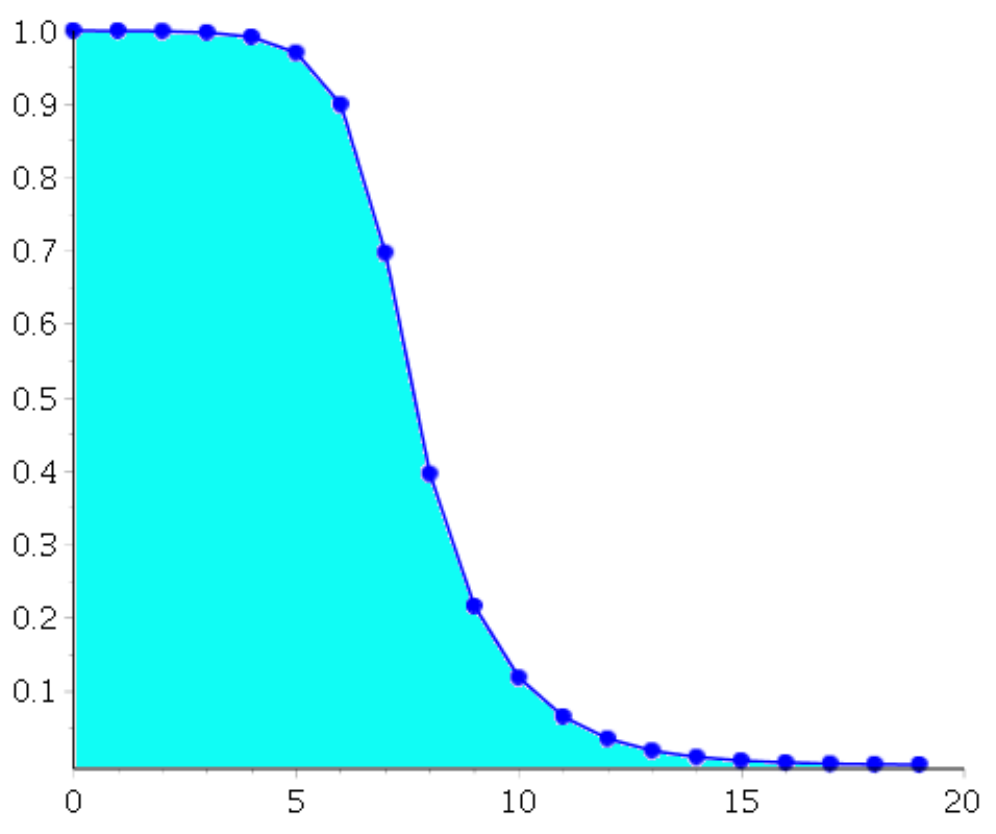}
\includegraphics[angle =0,height = 6.5cm]{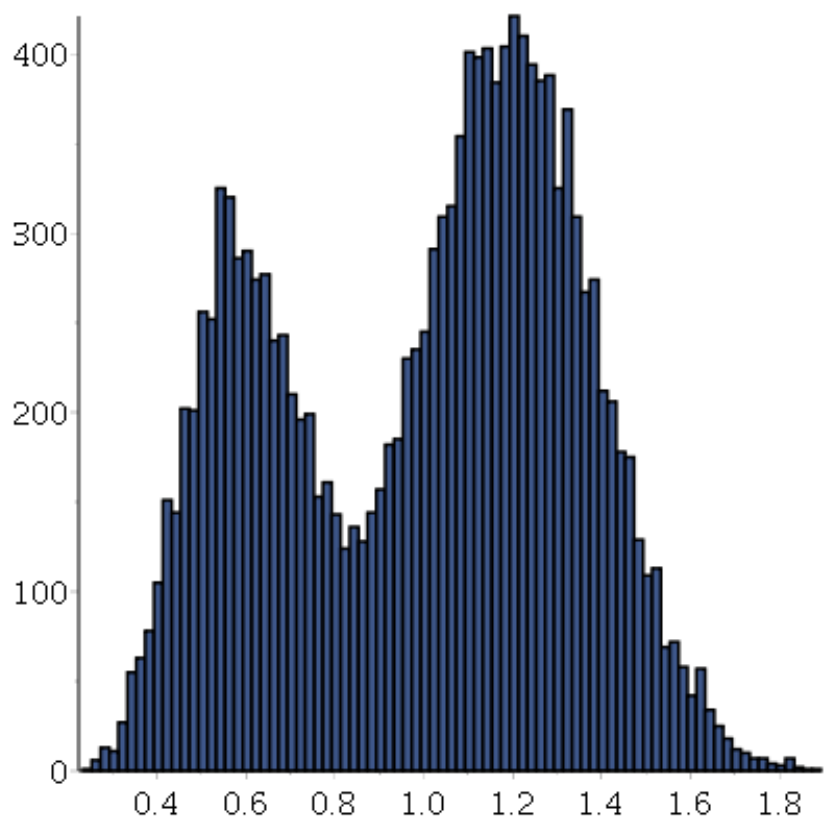}
\caption{Distance to equilibrium along time (left) and histogram of the vertex weights multiplied by $n$ under the equilibrium measure (right) for the random walk on a random digraph with $n=5000+5000+5000$ vertices with respective degrees $(d^+,d^-)=(3,2)$,  $(3,4)$ and  $(4,4)$.}
\label{fig:cutoff}
\end{center}
\end{figure}

\section{Introduction}
Given two sequences of positive integers $(d^-_i)_{1\leq i\leq n}$ and $(d^+_i)_{1\leq i\leq n}$ with equal sum $m$, we construct a directed multigraph $G$ with in-degrees $(d^-_i)_{1\leq i\leq n}$  and out-degrees  $(d^+_i)_{1\leq i\leq n}$ as follows. We formally equip each vertex $i\in V:=\{1,\ldots,n\}$ with a set $E^+_i$ of $d_i^+$ \emph{tails} and a set  $E^-_i$ of $d^-_i$ \emph{heads}. We then choose a tail-to-head bijection $\omega\colon \bigcup_{i}E_i^+\to\bigcup_{i}E_i^-$ (the \emph{environment}), and interpret each coordinate $\omega(e)=f$ as an arc $ef$ from the vertex of $e$ to that of $f$ (loops and multiple edges are allowed). Our interest is in the \emph{Random Walk} on the resulting directed graph (digraph) $G$, i.e. the discrete-time Markov chain with state space $V$ and transition matrix
\begin{eqnarray*}
P(i,j) & := & \frac{1}{d^+_i}\,\textrm{card}\left\{e\in E_i^+\colon \omega(e)\in E_j^-\right\}.
\end{eqnarray*}
Starting from $i\in V$, the probability that the walk is at $j\in V$ at time $t\in\dN$ can be expanded as
\begin{eqnarray}
\label{eq:expand}
P^t(i,j) & = & \sum_{\mathfrak p\in\cP^t_{ij}}\w(\mathfrak p),
\end{eqnarray}
where the sum ranges over all \emph{directed paths} $\mathfrak p$ of length $t$ from $i$ to $j$, i.e. sequences of arcs $\mathfrak p= (e_1f_1,\ldots,e_tf_t)$ with $e_k\in E^+_{i_{k-1}}$, $f_k\in  E^-_{i_k}$ for some sequence $i_k\in V$ such that $i_0=i$ and $i_t=j$, and where the \emph{weight}  is defined by 
\begin{eqnarray*}
\w(\mathfrak p) & := & \frac{1}{d^+_{i_0}\cdots d^+_{i_{t-1}}}.
\end{eqnarray*}
As long as $G$  strongly connected (in the sense that for every $i,j\in V$ there exists a directed path from $i$ to $j$), the classical theory  guarantees that there is a unique probability measure $\pi_\star$ on $V$ which is invariant, i.e.  $\pi_\star P=\pi_\star$ (see, e.g., the book  \cite{levin2009markov}). From every starting state $i\in V$, one may then consider the \emph{total-variation distance from equilibrium at time t}: 
\begin{eqnarray}
\label{def:distance}
\cD_i(t) & := & \|P^t(i,\cdot)-\pi_\star\|_{\textsc{tv}}\ = \ \frac 12\sum_{j\in V}\left|P^t(i,j)-\pi_\star(j)\right|\ \in \ [0,1].
\end{eqnarray}
The aim of this paper is to investigate the profile of the decreasing functions $t\mapsto \cD_i(t),i\in V$ under the \emph{configuration model}, i.e. when the environment $\omega$ is chosen uniformly at random from the $m!$ possible choices. This turns $G$, $P$, $\pi_\star$ and the $(\cD_i)_{i\in V}$ into random objects, parametrized by the degrees $(d^\pm_i)_{1\leq i\leq n}$. In order to study large-size asymptotics, we let all quantities implicitly depend on $n$ and consider the limit as $n\to\infty$. We restrict our attention to the \emph{sparse regime}, where
\begin{eqnarray}\label{assume:delta}
\delta:=\min_{1\leq i\leq n} d_i^\pm  \geq 2 & \textrm{ and } & \Delta:=\max_{1\leq i\leq n} d_i^\pm  =  \cO(1).
\end{eqnarray}
Note in particular that $m=\Theta(n)$. The requirement on $\delta$ guarantees that $G$ is strongly connected with high probability (see, e.g.,  \cite{MR2056402}). In particular, the equilibrium measure $\pi_\star$ is unique with high probability. Here and below, with high probability (w.h.p.) means with probability tending to $1$ as $n\to\infty$.  The mixing time of the walk turns out to be determined by a simple statistics, namely the mean logarithmic out-degree of the end-point of a uniformly chosen head. More precisely, define 
\begin{eqnarray*}\label{eq:mu}
\mu := \frac 1m\sum_{i=1}^nd_i^-\ln d_i^+ & \textrm{ and } & t_\star:=\frac{\ln n}{\mu}.
\end{eqnarray*}
Our first result is that $t_\star$ steps are necessary and sufficient for the walk to mix, regardless of the initial vertex. More precisely, as the number of iterations $t$ approaches $t_\star$, the distance from equilibrium undergoes a sharp transition  visible on Figure \ref{fig:cutoff} and known as a \emph{cutoff phenomenon}. 
\begin{theorem}[Cutoff at time $t_\star$] 
\label{th:cutoff}
For $t=\lambda t_\star+o(t_\star)$ with fixed $\lambda\geq 0$, we have 
\begin{eqnarray*}
\lambda < 1& \ \Longrightarrow \ & \min_{i\in V}\cD_i\left(t\right)  \ \xrightarrow[n \to\infty]{\PP} \ 1,\\
\lambda > 1& \ \Longrightarrow \ & \max_{i\in V}\cD_i\left(t\right)  \ \xrightarrow[n \to\infty]{\PP} \ 0,
\end{eqnarray*}
where $\xrightarrow[]{\PP}$ denotes convergence in probability. 
\end{theorem}

In view of Theorem \ref{th:cutoff}, it is tempting to ``zoom in'' around the cutoff point $t_\star$ until the details of the abrupt transition from $1$ to $0$ become visible. The appropriate window-width turns out to be 
\begin{eqnarray*}\label{eq:sigma}
w_\star := \frac{\sigma\sqrt{\ln n}}{\mu^{3/2}}, &  \textrm{ where } & \sigma^2  :=   \frac 1m\sum_{i=1}^nd_i^-\left(\ln d_i^+-\mu\right)^2.\end{eqnarray*}
Remarkably, the graph of the function $t\mapsto \cD_i(t)$ inside this window  approaches a universal shape, independent of the initial position $i$ and the precise degrees: the gaussian tail function. 

\begin{theorem}[Inside the cutoff window]
\label{th:window}
Assume that the variance $\sigma^2$ is asymptotically non-degenerate in the following weak sense:\begin{eqnarray}
\label{assume:sigma}
\sigma^2 & \gg & \frac{(\ln\ln n)^2}{\ln n}.
\end{eqnarray}
Then, for $t= t_\star+\lambda w_\star+o(w_\star)$ with $\lambda\in\RR$ fixed,  we have
\begin{eqnarray*}
\max_{i\in V}\left| \cD_i(t)-\frac{1}{\sqrt{2\pi}}\int_{\lambda}^\infty e^{-\frac{u^2}{2}}\D u\right | & \xrightarrow[n\to\infty]{\PP} & 0.
\end{eqnarray*}
\end{theorem}
In contrast, our third main result asserts that the mixing time is reduced from $t_\star$ to constantly many steps when starting from a more spread-out distribution, such as the \emph{in-degree distribution}:
\begin{eqnarray}
\label{eq:indegree}
\pi_{0}^{-}(i) & := & \frac{d^-_i}{m},\qquad i\in V.
\end{eqnarray}
For future reference, the \emph{out-degree distribution} is naturally defined by changing the  $-$ to $+$ above.
\begin{theorem}[Exponential convergence from a uniform head]
\label{th:uniform}
Fix $t\in\dN$ and set $\pi_{t}^{-}:=\pi_{0}^{-} P^t$. Then,
\begin{eqnarray*}
4\|\pi_{t}^{-}-\pi_\star\|_{\textsc{tv}}^2 & \leq &  \frac{n\left(\gamma-1\right)}{m(1-\varrho)}\varrho^t+o_\PP(1),
\end{eqnarray*}
where $o_\PP(1)$ denotes a term that tends to $0$ in probability as $n\to\infty$, and where $\varrho,\gamma$ are  defined by
\begin{eqnarray*}
\varrho \ := \ \frac{1}{m}\sum_{i=1}^n\frac{d_i^-}{d^+_i}
 &  \textrm{ and }& 
 \gamma \ := \ \frac{1}{m}\sum_{i=1}^n\frac{\left(d_i^-\right)^2}{d_i^+}.
\end{eqnarray*}
\end{theorem}
Note that $\varrho\leq\frac{1}{\delta}$ and that the constant in front of $\varrho^t$ is less than $\Delta$; see \eqref{assume:delta}. Thus, the convergence to equilibrium occurs exponentially fast, uniformly in $n$. Note also that $\gamma\geq 1$ by the Cauchy-Schwarz inequality, with equality if and only if 
the graph is Eulerian, that is $d_i^+=d_i^-$ for all $i\in V$. In this very special case, it is not hard to see that $\pi_{0}^{-}P=\pi_{0}^{-}$: the equilibrium measure $\pi_\star$ is nothing but the in-degree distribution $\pi_{0}^{-}$ itself, in agreement with the fact that the upper-bound given by Theorem \ref{th:uniform} vanishes.

One surprising implication of Theorem \ref{th:uniform} is that the equilibrium mass $\pi_\star(i)$  assigned to a vertex $i\in V$ is essentially determined by the (backward) local neighbourhood of $i$ only. By combining this with a simple branching process approximation (see Section \ref{sec:martingale}), we obtain rather precise asymptotics for the equilibrium measure. 
Specifically, our last main result concerns 
the empirical distribution 
\begin{eqnarray}\label{emp1}
\psi_n := \frac1n\sum_{i=1}^n\delta_{n\pi_\star(i)},
\end{eqnarray}
of the numbers $\{n\pi_\star(i)\}_{i\in V}$ (as depicted in Figure \ref{fig:cutoff}).
Clearly, $\psi_n$ is a random probability measure on $[0,\infty)$ satisfying 
\begin{eqnarray}\label{emp2}
\int_0^\infty x\,\psi_n(dx) = 1. 
\end{eqnarray}
Theorem \ref{th:measure} below asserts that $\psi_n$ concentrates around the deterministic probability measure $\cL$, defined as the law of the 
random variable
\begin{eqnarray}
\label{eq:mstar}
M_\star  & {=} & \frac{n}{m}\sum_{k=1}^{d^-_\cI}Z_k,
\end{eqnarray}
where $\cI$ is a random vertex, uniformly distributed on $V$, and $(Z_k)_{k\geq 1}$ is an independent set of  i.i.d.\ mean-one random variables with common law determined by the distributional fixed-point equation
\begin{eqnarray}
\label{rde}
Z_1 & \stackrel{d}{=} & \frac{1}{{d}^+_\cJ}\sum_{k=1}^{{d}^-_\cJ}Z_k,
\end{eqnarray}
where the random vertex $\cJ$ has the out-degree distribution $\pi_0^+$ and is independent of $({Z}_k)_{k\geq 1}$.
This recursive distributional equation has been extensively studied, $Z_1$ being a special case of a random variable stable by weighted sum. Such self-similar variables appear notably in connections with Mandelbrot's multiplicative cascades and branching random walks. A more general version of \eqref{rde} has been studied by R{\"o}sler \cite{MR1176497}, Liu \cite{MR1450934,MR1439972,MR1741808,MR1847093} and  Barral \cite{MR1717530,MR3275997}. Among others, these references provide detailed results concerning the uniqueness of the solution $Z_1$, its left and right tails, its positive and negative moments, its support, and even its absolute continuity w.r.t. Lebesgue's measure. 
For example, in the non-Eulerian case (${d}^+_i \ne {d}^-_i$ for some $1 \leq i \leq  n$),  Liu \cite[Theorem 2.3]{MR1847093} implies that the variable $Z$ of mean one which satisfies \eqref{rde}  is absolutely continuous.

To state our result, we recall that the $1-$Wasserstein (or Kantorovich-Rubinstein) distance between two probability measures   $\cL_1,\cL_2$ on $\dR$ (see, e.g., \cite{villani}[Chapter 6])  is defined as
\begin{eqnarray*}
\cW\left(\cL_1,\cL_2\right) & = & \sup_{f}\left|\int_{\dR}f{\D} \cL_1- \int_{\dR}f{\D} \cL_2\right|,
\end{eqnarray*}
where the supremum runs over all $f\colon\dR\to\dR$ satisfying $|f(x)-f(y)|\leq |x-y|$ for all $x,y\in\dR$.
\begin{theorem}[Structure of the equilibrium measure]Let $\cL$ be the law of $M_\star$ from \eqref{eq:mstar}-\eqref{rde}, and let $\psi_n$ be the empirical distribution of $\{n\pi_\star(i)\}_{i\in V}$ defined in \eqref{emp1}. Then,
\label{th:measure}
\begin{eqnarray*}
\cW\left(\psi_n,\cL\right) & \xrightarrow[n\to\infty]{\PP} & 0.
\end{eqnarray*}
\end{theorem}

\begin{remark}[Conservation of mass]\label{rem_invmeas}
The fact that the random empirical measure $\psi_n$ approaches the deterministic law $\cL$ in the $1-$Wasserstein sense rather than just in the usual weak sense is important, as it ensures that the first moment \eqref{emp2} is conserved, i.e., that $\cL$ has mean $1$. This removes the scale indeterminacy inherent to equation (\ref{rde}), and rules out the possibility that a non-vanishing part of the mass of $\pi_\star$ concentrates on a negligible fraction of the state space. Note, however, that $\cL$ does not capture the asymptotics of exceptional values such as $\max_{i\in[n]}\pi_\star(i)$ or $\min_{i\in[n]}\pi_\star(i)$. See the recent preprint \cite{2015arXivRout} for estimates on these extremes in the special case where all out-degrees are equal.
\end{remark}

\begin{remark}[Dependence on $n$]
Note that the deterministic measure $\cL$ actually depends on $n$, as did the quantities $\mu,\sigma^2, \gamma,\varrho$ appearing in the above theorems: we have chosen to express all our approximations directly in terms of the true degree sequence $(d_i^\pm)_{1\leq i\leq n}$ (which we view as our input parameter), rather than loosing in generality by artificially assuming the weak convergence of the empirical degree distribution $\frac{1}{n}\sum_{i=1}^n\delta_{(d_i^-,d_i^+)}$ to some $n-$independent limit. Of course, any external assumption on the $n\to\infty$ behaviour of the degrees can then  be ``projected" onto the $n-$dependent constants provided by our results to yield \emph{bona fide} convergence results, if needed.
\end{remark}

\section{Related work}

The phase transition described in Theorem \ref{th:cutoff} is an instance of the celebrated \emph{cutoff phenomenon}, first singled out in the early 1980's in the context of card shuffling by Diaconis, Shahshahani and Aldous \cite{diaconis1981generating,aldous1983mixing,aldous1986shuffling}. This remarkable discontinuity in the convergence to equilibrium of an ergodic Markov chain has since then been identified in a variety of contexts, ranging from random walks on groups to interacting particle systems. We refer the reader to  \cite{diaconis1996cutoff,saloff2004random,chen2008cutoff} for more details. Perhaps surprisingly, the emergence of the gaussian shape inside the cutoff window in our Theorem \ref{th:window} is not isolated: the very same feature has been observed in a few unrelated models, such as the random walk on the $n-$dimensional hypercube \cite{MR1068491}, or the simple-exclusion process on the circle \cite{2015arXiv150200952L}.

Motivated by applications to real-world networks  (see, e.g., the survey by Cooper \cite{cooper2011random} and the references therein), the mixing properties of random walks on large but finite random graphs have recently become the subject of many investigations. The attention has been mostly restricted to the undirected setting, where the in-degree distribution is reversible, and therefore stationary. In particular, Frieze and Cooper have studied the \emph{cover time} (i.e., the expected time needed for the chain to visit all states) of various random graphs \cite{MR2157821,MR2283218,MR2290325,MR2422388}, and analyzed the precise component structures induced by the walk on random regular graphs \cite{cooper2014vacant}. Bounds for the mixing time on the largest component of the popular Erd\H os--Renyi model have also been obtained by various authors, in both the critical and super-critical connectivity  regime \cite{MR2435849,MR3252922,MR2428978,MR2962084}. 

More directly related to our work is the  inspiring paper \cite{lubetzky2010cutoff}  by Lubetzky and Sly, which establishes the cutoff phenomenon and determines its precise window and shape for  the simple and the non-backtracking random walks on random regular graphs. The results therein were very recently generalized to all non-bipartite regular Ramanujan graphs by Lubetzky and Peres \cite{2015arXivRamanujan}, thereby confirming a long-standing conjecture of Peres \cite{peresamerican}. In  \cite{2015arXiv}, Berestycki, Lubetzky, Peres and Sly establish the cutoff phenomenon on the Erd\H os--Renyi model and on the more general configuration model, for both the simple and non-backtracking random walks. The latter case was  simultaneously and independently addressed by Ben-Hamou and Salez \cite{2015arXivNBRW}, who additionally determine the precise second-order behaviour inside the cutoff window. 

In contrast, very little is known about random walks on random directed graphs, and the present paper seems to provide the first proof of a cutoff phenomenon in this setting. The failure of the crucial \emph{reversibility} property makes many of the ingredients used in the above works unavailable. In fact, even understanding the equilibrium measure  constitutes an important theoretical challenge, with applications to link-based ranking in large databases (see, e.g., \cite{2014arXiv14097443} and the references therein). In \cite{MR2885424}, Cooper and Frieze consider the random digraph on $n$ vertices formed by independently placing an arc between every pair of vertices with probability $p=\frac{d\ln n}{n}$, where $d>1$ is fixed while $n\to\infty$. In this regime, they prove that the equilibrium measure is asymptotically close to the in-degree distribution. The more recent work \cite{2015arXivRout} by Addario-Berry, Balle and Perarnau provides precise estimates on the extrema of the equilibrium measure in the special case where all out-degrees are equal.  
As mentioned in Remark \ref{rem_invmeas}, such extremal values are not captured by our Wasserstein approximation for the empirical measure $\frac 1n\sum_{i=1}^n\delta_{n\pi_\star(i)}$, and understanding them in the general setting considered here remains an interesting open problem.

\section{Proof outline and main ingredients}
\subsection{Dealing with an unknown equilibrium measure}
\label{subsec:proxypi}
One difficulty in controlling the distance to equilibrium is that $\pi_\star$ is not known a priori. We will have to work instead with the proxy $\wpi:=\pi_{0}^{-} P^h$,  where $\pi_{0}^{-}$ is the in-degree distribution and
\begin{eqnarray}
\label{def:h}
h & := & \left\lfloor\frac{\ln n}{10\ln \Delta}\right\rfloor.
\end{eqnarray}
Establishing Theorems \ref{th:cutoff} and \ref{th:window} with  $$\widetilde{\cD}_i(t):=\|P^t(i,\cdot)-\wpi\|_{\textsc{tv}}$$ instead of $\cD_i(t)$ is actually sufficient. Indeed, it ensures in particular that for (say) $t\geq 2t_\star$, 
\begin{eqnarray*}
\max_{i\in V}\|P^t(i,\cdot)-\wpi\|_{\textsc{tv}} & \xrightarrow[n\to\infty]{\PP} & 0.
\end{eqnarray*}
Now, by convexity, this maximum automatically extends to all initial distributions on $V$. In particular, one may start from the invariant measure $\pi_\star$ itself: since $\pi_\star P^t=\pi_\star$, we obtain
\begin{eqnarray}
\label{eq:proxy}
\|\pi_\star-\wpi\|_{\textsc{tv}} & \xrightarrow[n\to\infty]{\PP} & 0.
\end{eqnarray}
By the triangle inequality, we finally deduce that
\begin{eqnarray*}
\sup_{i\in V,\,t\in\dN}\left|\widetilde{\cD}_i(t)-\cD_i(t)\right| & \xrightarrow[n\to\infty]{\PP} & 0,
\end{eqnarray*}
so that the conclusions obtained for $\widetilde{\cD}_i$ in Theorems 1 and 2 are automatically transferred  to $\cD_i$.  
\subsection{Sequential generation and tree-like structure} 
\label{sec:treelike}
An elementary yet crucial observation about the uniform environment $\omega$ is that it can be generated sequentially, starting with all heads and tails unmatched, by repeating $m$ times the following steps: 
\begin{enumerate}
\item an unmatched tail $e$ is selected according to some priority rule;
\item an unmatched head $f$ is chosen uniformly at random;
\item $e$ is matched with $f$ to form the arc $ef$, that is, $\omega(e):=f$.
\end{enumerate}
The resulting bijection $\omega$ is uniform, regardless of which particular priority rule is used. 
In the sequel, we shall intensively exploit this degree of freedom to simplify the analysis of  the environment. In this respect, the following observation will prove useful. Let us say that a \emph{collision} occurs whenever a head gets chosen, whose end-point $i$ was already \emph{alive} in the sense that some tail $e\in E^+_i$ or head $f\in E^-_i$ had previously been chosen. Since less than $2k$ vertices are alive when the $k^{\textrm{th}}$ head gets chosen, less than $2\Delta k$ of the $m-k+1$ possible choices can result in a collision. Thus, the conditional chance that the $k^{\textrm{th}}$ arc causes a collision, given the past, is less than
$\frac{2\Delta k}{m-k+1}$:
\begin{lemma}[Collisions are rare]\label{lm:collisions}
Let $1\leq k\leq m$. The number $Z_k$ of collisions caused by the first $k$ arcs is stochastically dominated by a Binomial$\left(k,\frac{2\Delta k}{m-k+1}\right)$ random variable. In particular, 
\begin{eqnarray*}
\PP\left(Z_k\geq 1\right) \ \leq \ \frac{2\Delta k^2}{m-k+1} & \textrm{ and } & \PP\left(Z_k\geq 2\right) \ \leq \ \frac{2\Delta^2 k^4}{(m-k+1)^2}.
\end{eqnarray*}
\end{lemma}
Here is one application: the \emph{forward ball} of radius $t$ around $i\in V$ is the subgraph $\cB^+(i,t)\subseteq G$ induced by the directed paths of length $t$ from $i$. It can be sequentially generated by giving priority to those unmatched tails $e$ that currently lie at minimal distance from $i$, until this minimal distance exceeds $t$. At most $k=\Delta+\cdots+\Delta^t$ edges are formed by then, and each collision corresponds to the formation of a transverse arc violating the directed-tree structure of $\cB^+(i,t)$.  Choosing $t=2h$ with $h$ defined in (\ref{def:h}) ensures that $\PP\left(Z_k\geq 2\right)=o\left(\frac 1n\right)$, uniformly in $i\in V$. We may thus take a union-bound and conclude that with high probability, $G$ is locally tree-like in the following sense:
\begin{eqnarray*}
\forall i\in V, & & \cB^+(i,2h)  \textrm{ is either a directed tree, or a directed tree with an extra arc}. 
\end{eqnarray*}

Combining this with the fact that all out-degrees are at least $2$, it is not difficult  to deduce the following result. 
Let $V_\star$ denote the set of vertices $i\in V$ such that $\cB^+(i,h)$ is a directed tree. 
 \begin{proposition}
 \label{pr:treelike}
With high probability, 
\begin{eqnarray*}
\forall i\in V\,,\;\forall \ell\in\dN\,, \qquad P^\ell(i,V\setminus V_\star) & \leq & 2^{-\ell\wedge h}.
\end{eqnarray*}
\end{proposition}
In other words, $V_\star$ is typically attained in constant time by the random walk, from any initial position. This will be very helpful, as the walk is much easier to control when starting from $V_\star$.  This takes after the good idea of `roots' that was introduced by Lubetzky and Sly \cite{lubetzky2010cutoff} in the case of regular graphs, and then reused  in  \cite{2015arXiv,2015arXivNBRW} for graphs with more general degrees.

\begin{remark}[Sampling with or without replacement]
\label{rk:coupling}
Instead of choosing heads uniformly among the remaining unmatched ones, we may choose uniformly from all heads, and retry if the chosen head was already matched. The chance that this occurs within the first $k$ steps is
$p = 1-\prod_{i=0}^{k-1}\frac{m-i}{m}$. This creates a coupling between the first $k$ chosen heads and an i.i.d. sample from the uniform distribution on all heads, under which the two sequences coincide with probability $1-p$. In the analysis of the sequential  process, we may thus replace the former by the latter at a total-variation cost $p\leq k^2/m$. We may even stop the coupling as soon as a head is sampled, whose end-point was already alive: this ensures the absence of collision, while multiplying the total-variation cost $p$ by a factor at most $\Delta$. Note that the end-point of a uniform head has the in-degree distribution.
\end{remark}
\subsection{Typical path weights}

 At a high level, the cutoff phenomenon around time $t_\star=\frac{\ln n}{\mu}$ described in Theorem 1 can be understood as the consequence of the following two key principles: 
\begin{enumerate}
\item[(i)] Trajectories whose weight exceeds  $\frac 1n$ constitute the essential obstruction to mixing, in the precise sense that their total weight is roughly equal to the distance to equilibrium. 
\item[(ii)] From the point-of-view of the walk, most trajectories of length $t$ have weight $e^{-\mu t+\cO(\sqrt{t})}$. 
\end{enumerate}
Thus, the essence of the cutoff phenomenon for the random walk on $G$ lies in a sharp concentration phenomenon for the weight of the paths seen by the walk. 
To formalize this idea, let 
\begin{eqnarray*}
\cQ_{i,t}(\theta) & := & \sum_{j\in V}\sum_{\mathfrak p\in\cP^t_{ij}}\w(\mathfrak p){\bf 1}_{\w\left({\mathfrak p}\right)> \theta}
\end{eqnarray*}
denote the quenched probability  that a random walk of length $t$ starting at $i$ follows a path whose weight exceeds $\theta$.  The above two claims can then be given the following precise meanings. 
\begin{proposition}[High-weight paths determine the distance to equilibrium]For any $t=t(n)$, 
\label{pr:distance}
\begin{eqnarray*}
\min_{i\in V}\wcD_i(t) & \geq & \min_{i\in V}\cQ_{i,t}\left(\frac{\ln^3n}{n}\right)-o_\PP(1)\\
\max_{i\in V}\wcD_i(t) & \leq & \max_{i\in V}\cQ_{i,t}\left(\frac{1}{n\ln^3n}\right)+o_\PP(1).
\end{eqnarray*}
\end{proposition}
\begin{proposition}[Most paths of length $t$ have weight $e^{-\mu t+\cO(\sqrt{t})}$] Assume that $t=\Theta(\ln n)$, and let $\theta$ depend arbitrarily on $n$. 
\label{pr:weight}
\begin{enumerate}
\item If $\frac{\mu t+\ln \theta}{\sqrt t} \to +\infty$ as $n\to\infty$, then 
\begin{eqnarray*}
\max_{i\in V}\cQ_{i,t}(\theta) & \xrightarrow[n\to\infty]{\PP}  & 0.
\end{eqnarray*}
\item If $\frac{\mu t+\ln \theta}{\sqrt t} \to -\infty$ as $n\to\infty$, then 
\begin{eqnarray*}
\min_{i\in V}\cQ_{i,t}(\theta)&  \xrightarrow[n\to\infty]{\PP} &  1.
\end{eqnarray*}
\item If $\frac{\mu t+\ln \theta}{\sigma\sqrt t} \to\lambda\in\dR$ as $n\to\infty$ and if assumption (\ref{assume:sigma}) holds, then 
\begin{eqnarray*}
\max_{i\in V}\left|\cQ_{i,t}(\theta)-\frac{1}{\sqrt{2\pi}}\int_{\lambda}^{\infty} e^{-\frac{u^2}{2}}\D u \right| & \xrightarrow[n\to\infty]{\PP} & 0.
\end{eqnarray*}
\end{enumerate} 
\end{proposition}
Those two results clearly imply Theorems \ref{th:cutoff} and \ref{th:window}, and their proofs occupy much of the paper. The detailed organization is as follows. 

Section \ref{sec:lower} is devoted to the proof of the lower bound in Proposition \ref{pr:distance}. The latter is based on the following simple observation: if the typical trajectory seen by a random walk of length $t$ has a probability much larger than $1/n$ to be actually followed,  then the random walk is essentially confined to $o(n)$ distinct trajectories and is therefore far from being mixed. A similar argument was already used in \cite{2015arXivNBRW}. Note, however, that this argument implicitly assumes that the equilibrium distribution is well spread out over the state space, a fact which was obvious in \cite{2015arXivNBRW} but is not at all clear here (except in the very special Eulerian case, where $\pi_\star$ reduces to the trivial in-degree distribution).

Section \ref{sec:weight} is devoted to the proof of Proposition \ref{pr:weight}. We first use Proposition \ref{pr:treelike} to reduce our task to that of controlling the chain from a slightly more spread-out initial law, namely $P^\ell(i, \cdot)$ with $i$ in $V_\star$ and $1\ll \ell \ll t_\star$. From such an initial condition, we may build on Remark \ref{rk:coupling} to couple the sequence of vertices visited by the random walker with i.i.d. samples from the in-degree distribution at a cost that is sufficiently low to allow for a union bound over all possible origins $i \in V_\star$, see Lemma \ref{lm:coupling}. The lack of time-reversibility is not a problem at all here, and a similar argument was actually used in \cite{2015arXiv,2015arXivNBRW} for  undirected graphs.

The upper bound in Proposition \ref{pr:distance} amounts to establishing a uniform  estimate of the form $$P^t (i,j) \leq {\pi_{h}^{-}} (j) + \veps(i,j),$$ with $\sum_j \veps(i,j) \leq  \cQ_{i,t}\left(\frac{1}{n\ln^3n}\right)+o_\PP(1),$ which is much more involved. This is done in Section \ref{sec:upper}, where we prove that the set of trajectories followed by the random walk is concentrated on a suitably defined collection of {\em nice paths} whose probabilities can be accurately estimated by applying a powerful concentration inequality for random bijections due to Chatterjee \cite{chatterjee2007stein}. We note that our definition of a nice path (Definition \ref{def:nice}) involves the local environment around both the origin and the destination, in a non-symmetric way. This direct consequence of the lack of time-reversibility  constitutes a substantial difference with the settings considered in  \cite{lubetzky2010cutoff,2015arXiv,2015arXivNBRW}. In particular, the exploration of the \emph{backward ball} around the destination requires a careful, specific treatment, and the technical assumption $\Delta=\cO(1)$ is used heavily at this point. Note that the difficulty persists even in the Eulerian case, despite the simple explicit form for $\pi_\star$.

Finally, once (\ref{eq:proxy}) is established, the branching process approximation for $G$ developed in section \ref{sec:martingale} quickly leads to  the proof of Theorems \ref{th:uniform} and \ref{th:measure} in sections \ref{sec:uniform} and \ref{sec:measure}.

\section{Proof of the lower-bound in Proposition \ref{pr:distance}}
\label{sec:lower}
In this section, we prove the easy half of Proposition \ref{pr:distance}: the lower-bound. 
Fix the environment $\omega$, a probability measure $\pi$ on $V$, $t\in\dN$, $\theta\in(0,1)$ and $i,j\in V$. Consider the inequality
\begin{eqnarray}
\label{ineq}
P^t(i,j) & \geq & \sum_{\mathfrak p\in\cP^t_{ij}}\w(\mathfrak p){\bf 1}_{\w(\mathfrak p)\leq \theta}.
\end{eqnarray}
If equality holds in \eqref{ineq}, then clearly 
\begin{eqnarray*}
\pi(j)-\sum_{\mathfrak p\in\cP^t_{ij}}\w(\mathfrak p){\bf 1}_{\w(\mathfrak p)\leq \theta} & \leq & \left[\pi(j)-P^t(i,j)\right]_+,\end{eqnarray*}
where $[x]_+:=\max(x,0)$ is the positive part of $x$. On the other-hand, if the inequality (\ref{ineq}) is strict, then 
 there must exist $\mathfrak p\in\cP^t_{ij}$ such that $\w(\mathfrak p)>\theta$, implying that $P^t(i,j)>\theta$, and hence that
 \begin{eqnarray*}
\pi(j)-\sum_{\mathfrak p\in\cP^t_{ij}}\w(\mathfrak p){\bf 1}_{\w(\mathfrak p)\leq \theta} & \leq & \pi(j){\bf 1}_{P^t(i,j)>\theta}.
\end{eqnarray*}
Thus, in either case, 
\begin{eqnarray*}
\pi(j)-\sum_{\mathfrak p\in\cP^t_{ij}}\w(\mathfrak p){\bf 1}_{\w(\mathfrak p)\leq \theta} & \leq & \left[\pi(j)-P^t(i,j)\right]_+ + \pi(j){\bf 1}_{P^t(i,j)>\theta}.
\end{eqnarray*}
Summing over all $j\in V$ and invoking the Cauchy--Schwarz  inequality and the fact that there are less than $1/\theta$ indices $j$ such that $P^t(i,j)>\theta$, we get
\begin{eqnarray}
\cQ_{i,t}(\theta) & \leq & \|\pi-P^t(i,\cdot)\|_{\textsc{tv}} + \sqrt{\frac{1}{\theta}\sum_{j\in V}\pi^2(j)}.\label{eq:lbQt}
\end{eqnarray}
We now specialize to our random environment, with $\theta=\frac{\ln^3 n}{n}$ and $\pi=\wpi$. To conclude the proof, it remains to verify that the square-root term is $o_\PP(1)$. We will in fact show the stronger
\begin{eqnarray*}
\EE\left[\sum_{j\in V}(\wpi)^2(j)\right]  = \cO \left(\frac{\ln^2 n}{n} \right)& \ll & \theta.
\end{eqnarray*}
Since $\wpi=\pi_{0}^{-} P^h$, the left-hand-side may be interpreted as $\PP\left(X_h=Y_h\right)$, where conditionally on the environment, $(X_k)_{0\leq k\leq h}$ and $(Y_k)_{0\leq k\leq h}$ are two independent random walks starting from the in-degree distribution $\pi_{0}^{-}$, and where the average is taken over both the walks and the environment. To evaluate this annealed probability, we generate the walks together with the environment, forming arcs along the way, as we need them. Initially, all tails and heads are unmatched, and $X_0$ is chosen according to the in-degree distribution. Then, inductively, at every step $1\leq k\leq  h$ :
\begin{enumerate}
\item A tail $e$ of the vertex $X_{k-1}$ is chosen uniformly at random.
\item If $e$ is unmatched, it gets matched to a uniformly chosen unmatched head $f$, i.e. $\omega(e):=f$.
\item In either case, $\omega(e)$ is now well-defined, and we let $X_k$ be its end-point. 
\end{enumerate}
Once $(X_k)_{0\leq k\leq h}$ has been generated, we proceed similarly with $(Y_k)_{0\leq k\leq h}$. Note that at most $2h$ arcs are formed during this process. For the event $\{X_h=Y_h\}$ to be realized, either a collision must occur (this has probability at most $\frac{8\Delta h^2}{m-2h}$ by Lemma \ref{lm:collisions}), or $\{Y_0=X_0\}$ must occur (this has probability at most $\frac\Delta m$). We deduce that
\begin{eqnarray*}
\PP\left(X_h=Y_h\right) & \leq & \frac{8\Delta (h+1)^2}{m-2h}.
\end{eqnarray*}  
The right-hand side is $\cO(\frac{\ln ^2n}n)$. In view of our choice of $\theta$, this concludes the proof. 
\section{Proof of Proposition \ref{pr:weight}}
Recall that Proposition \ref{pr:weight} controls, uniformly in $i\in V$, the quenched probability $\cQ_{i,t}(\theta)$ that a random walk of length $t$ starting at $i$ follows a path whose weight exceeds $\theta$. We first restrict our attention to vertices $i$ in $V_\star$, and replace $\cQ_{i,t}(\theta)$ with the more stable spatial average
$$\overline{\cQ}_{i,t}(\theta):=\sum_{j\in V}P^\ell(i,j)\cQ_{j,t}(\theta),\qquad\textrm{ where }\qquad \ell:=\lfloor 3 \ln\ln n\rfloor.$$
Let $(D_k^+)_{k\geq 1}$ denote i.i.d. copies of $d^+_\cJ$ where $\cJ$ follows the in-degree distribution on $V$. More explicitly, $\PP\left(D_k^+=d\right)=\frac{1}{m}\sum_{i=1}^nd^-_i{\bf 1}_{(d^+_i=d)}$. For $\theta\in[0,1]$, define
\begin{eqnarray*}
q_t(\theta) & :=  & \PP\left(\prod_{k=1}^t\frac{1}{D_k^+}>\theta\right). 
\end{eqnarray*}
\begin{lemma}For $t=\Theta(\ln n)$ and $\theta$ depending arbitrarily on $n$, 
\label{lm:coupling}
\begin{eqnarray*}
\max_{i\in V_\star}\left|\overline{\cQ}_{i,t}(\theta)-q_t(\theta)\right| & \xrightarrow[n\to\infty]{\PP} & 0. 
\end{eqnarray*}
\end{lemma}
\begin{proof}
Given the environment  $\omega$, consider $r=\lfloor\ln^2 n\rfloor$ independent walks of length $\ell+t$ starting at $i\in V$. Let $A_i$ be the event that the union of their  trajectories up to time $\ell$ (ignoring edge multiplicities) forms a tree and that their trajectories after time $\ell$  all have weights $>\theta$. Since $h > \ell$, ${\bf 1}_{i\in V_\star}\left(\overline{\cQ}_{i,t}(\theta)\right)^r$ is a lower-bound for the quenched probability of $A_i$. Averaging over $\omega$ and using Markov's inequality, we get for $\varepsilon>0$, 
\begin{eqnarray}
\label{eq:markov}
\PP\left(i\in V_\star, \overline{\cQ}_{i,t}(\theta) \geq q_t(\theta)+\varepsilon\right) & \leq & \frac{\PP(A_i)}{\left(q_t(\theta)+\varepsilon\right)^r}.
\end{eqnarray}
To evaluate the annealed probability $\PP(A_i)$, we may generate the $r$ walks one after the other together with the environment, forming arcs only when they are traversed by the walks, as we did in the previous section. Given that the first $k-1$ walks do satisfy the requirement, the conditional chance that the $k^{\rm{th}}$ walk also does is at most $q_t(\theta)+o(1)$ uniformly in $i\in V$ and $1\leq k <r$. Indeed,
\begin{itemize}
\item either the $k^\textrm{th}$ walk attains length $\ell$ before reaching an unmatched tail: thanks to the tree structure, there are at most $k-1< r$ possible trajectories to follow, and each has weight at most $2^{-\ell}$ by our assumption on the minimum degrees. Thus, the conditional chance of this scenario is less than
$r2^{-\ell}=o(1)$.
\item or the walk has reached an unmatched tail by time $\ell$: as explained in remark \ref{rk:coupling}, the remainder of the trajectory after the first unmatched tail can be coupled with an i.i.d. sample from the in-degree distribution on $V$ at a total-variation cost less than $\frac{\Delta r(t+\ell)^2}{m}=o(1)$. Thus, the conditional chance that the walk meets the requirement in that case is at most
$q_t(\theta)+o(1)$.
\end{itemize}
This proves that $\PP(A_i)\leq (q_t(\theta)+o(1))^r$, uniformly in $i\in V$. Inserting into (\ref{eq:markov}) and summing over all possible $i\in V$, we obtain the first half of the claim:  
\begin{eqnarray*}
\PP\left(\exists i\in V_\star, \overline{\cQ}_{i,t}(\theta) \geq q_t(\theta)+\varepsilon\right) & \xrightarrow[n\to\infty]{} & 0.
\end{eqnarray*}
Replacing ``{$>\theta$}'' by ``{$\leq \theta$}'' in the definitions of $A_i$,  $\overline{\cQ}_{i,t}(\theta)$ and $q_t(\theta)$ yields the other half.
\end{proof}
\begin{proof}[Proof of Proposition \ref{pr:weight}]
\label{sec:weight}Recalling the definitions of $ {\cQ}_{i,t}(\theta)$ and $ \overline{\cQ}_{i,t}(\theta)$, we have
\begin{eqnarray*}
\max_{i\in V}\cQ_{i,t}(\theta) & \leq &  \max_{i\in V_\star}\overline{\cQ}_{i,t-2\ell}(\theta2^\ell)+\max_{i\in V}P^\ell(i,V\setminus V_\star).
\end{eqnarray*}
Now, observe that $\overline{\cQ}_{i,t-s}(\theta)\leq \overline{\cQ}_{i,t}(\theta\Delta^{-s})$ for all $s\leq t$. Choosing $s=2\ell$, we obtain 
\begin{eqnarray}
\nonumber
\max_{i\in V}\cQ_{i,t}(\theta) 
& \leq &  \max_{i\in V_\star}\overline{\cQ}_{i,t}(\theta\Delta^{-2\ell})+\max_{i\in V}P^\ell(i,V\setminus V_\star)\\
\label{eq:inject1}& \leq &  q_t(\theta\Delta^{-2\ell})+o_\PP(1),
\end{eqnarray}
where in the second line we have used Lemma \ref{lm:coupling} and Proposition \ref{pr:treelike}. Similarly, we may write
\begin{eqnarray*}
\cQ_{i,t}(\theta) 
& \geq &  P^\ell(i,V_\star) \min_{j\in V_\star}\overline{\cQ}_{j,t-2\ell}(\theta\Delta^{2\ell}),
\end{eqnarray*}
and then use $P^\ell(i,V_\star)=1-P^\ell(i,V\setminus V_\star)$ and the fact that $\overline{\cQ}_{j,t}(\theta)$ is non-increasing in $t$ to obtain
\begin{eqnarray}
\nonumber
\min_{i\in V}\cQ_{i,t}(\theta) & \geq & \min_{i\in V_\star}\overline{\cQ}_{i,t}(\theta\Delta^{2\ell})-\max_{i\in V}P^\ell(i,V\setminus V_\star)\\
\label{eq:inject2}  & \geq & q_t(\theta\Delta^{2\ell})-o_\PP(1),
\end{eqnarray}
where the last inequality follows again from Lemma \ref{lm:coupling} and Proposition \ref{pr:treelike}. 
Recalling that $(\ln D_k^+)_{k\geq 0}$ are i.i.d. with mean $\mu$ and variance $\sigma^2$, we may now conclude as follows: 
\begin{itemize}
\item If $\frac{\mu t+\ln\theta}{\sqrt t}\to+\infty$ then $q_t(\theta)\to 0$, by Chebychev's inequality. 
\item If $\frac{\mu t+\ln\theta}{\sqrt t}\to -\infty$ then $q_t(\theta)\to 1$, by Chebychev's inequality again. 
\item If $\frac{\mu t+\ln\theta}{\sigma\sqrt t}\to \lambda\in\dR$ and $(\ref{assume:sigma})$ holds, then by Lindeberg's Central Limit Theorem (the law of $\ln D_1^+$ depends implicitly on $n$, but its support is uniformly bounded thanks to (\ref{assume:delta})),
\begin{eqnarray*}
q_t(\theta)\xrightarrow[n\to\infty]{} \frac 1{\sqrt{2\pi}}\int_{\lambda}^\infty e^{-\frac{u^2}2}\D u.
\end{eqnarray*}
\end{itemize}
In each case, the assumption on $\theta$ is insensitive to multiplication by $\Delta^{\pm 2\ell}$: indeed, this only shifts $\ln \theta$ by $\cO(\ln\ln n)$, which is negligible compared to $\sqrt{t}$, and even to $\sigma\sqrt t$ under the extra assumption  (\ref{assume:sigma}). Consequently, the same three conclusions hold with $\theta$ replaced by $\theta\Delta^{\pm 2\ell}$, and inserting into (\ref{eq:inject1}) and (\ref{eq:inject2})  concludes the proof of Proposition \ref{pr:weight}. 
\end{proof}
\section{Proof of the upper-bound in Proposition \ref{pr:distance}}
\label{sec:upper}
\subsection{General strategy}
We will estimate $P^t(i,j)$ by focusing on a certain collection of ``nice'' paths, defined as follows:
\begin{definition}\label{def:nice} A path ${\mathfrak p}$ of length $t$ starting at $i$ is \emph{nice} if it satisfies the following conditions:
\begin{enumerate}
\item $\w(\mathfrak p)\leq \frac{1}{n \ln^2 n}$ ;
\item the first $t-h$ steps of $\mathfrak p$ are contained in a certain subtree $\cT_i\subseteq G$, defined below;
\item the last $h$ steps of $\mathfrak p$ form the only path of length at most $h$ from its origin to destination. 
\end{enumerate}
\end{definition}
Let us fix the environment and consider the lower-bound $P^t_0(i,j)\leq P^t(i,j)$ obtained by restricting the sum  (\ref{eq:expand}) to these nice paths, i.e.
\begin{eqnarray*}
P^t_0 (i,j) & = & \sum \{ \w(\mathfrak p)  : \mathfrak p \hbox{ nice path of length $t$ from $i$ to $j$}\}.
\end{eqnarray*}
 Assume that this lower bound is small enough so that for all $i,j\in V$, 
\begin{eqnarray}
\label{main}
P^t_0(i,j) & \leq & (1+\varepsilon)\pi(j)+\frac{\varepsilon}{|V|},
\end{eqnarray}
for a given $\varepsilon>0$ and a given probability measure $\pi$ on $V$. Then,
\begin{eqnarray*}
\nonumber
\|\pi-P^t(i,\cdot)\|_{\textsc{tv}} & = & \sum_{j\in V}\left(\pi(j)-P^t(i,j)\right)_+
\\
\nonumber
& \leq & \sum_{j\in V}\left(\pi(j)(1+\varepsilon)+\frac{\varepsilon}{|V|}-P^t_0(i,j)\right)_+\ = \ 
  p(i)+2\varepsilon,
\end{eqnarray*}
where $p(i)=1-P_0^t(i,V)$ is the quenched probability that the path followed by the random walk of length $t$ starting at $i$ is not  \emph{nice}, and we have used \eqref{main} in the last identity. Our aim will thus be to establish the following result. 
\begin{proposition}
\label{pr:upper}
Fix $\varepsilon>0$ and assume that $t=t_\star+o(t_\star)$. Then with high probability,
\begin{enumerate}
\item every pair $(i,j)\in V^2$ satisfies (\ref{main}) with $\pi=\wpi$;
\item every $i\in V_\star$ satisfies $p(i)\leq \cQ_{i,t}\left(\frac{1}{n \ln^2 n}\right)+\varepsilon$.
\end{enumerate}
\end{proposition}
Let us show that Proposition \ref{pr:upper} implies the upper bound in Proposition \ref{pr:distance}.
Indeed, from Proposition \ref{pr:upper} we immediately deduce that for $t=t_\star+o(t_\star)$,
\begin{eqnarray}
\label{eq:distanceup}
\max_{i\in V_\star}\wcD_{i}(t) & \leq & \max_{i\in V_\star}\cQ_{i,t}\left(\frac{1}{n \ln^2 n}\right)+o_\PP(1).
\end{eqnarray}
This then extends to all initial states $i \in V$ as follows. Take $s=\lfloor\frac{\ln\ln n}{\ln\Delta}\rfloor$ and observe that 
\begin{eqnarray*}
\max_{i\in V}\wcD_{i}(t) & \leq & \max_{i\in V_\star}\wcD_{i}(t-s)+o_\PP(1)\\
& \leq & \max_{i\in V_\star}\cQ_{i,t-s}\left(\frac{1}{n \ln^2 n}\right)+o_\PP(1)\\
& \leq & \max_{i\in V_\star}\cQ_{i,t}\left(\frac{1}{n \ln^3 n}\right)+o_\PP(1),
\end{eqnarray*}
where we have used Proposition \ref{pr:treelike} and the fact that $\cQ_{i,t-s}\left(\theta\right)\leq \cQ_{i,t}\left({\theta}{\Delta^{-s}}\right)$. This proves the  upper bound in Proposition \ref{pr:distance} for all $t=t(n)$ of the form $t=t_\star+o(t_\star)$. 
On the other hand, using Proposition \ref{pr:weight} it is not difficult to extend these bounds to arbitrary $t=t(n)$.

\subsection{Construction of the subtree $\cT_i$ and proof of the first part in Proposition \ref{pr:upper}}
We shall sample the pairings determining the environment $\omega$ by a three-stage procedure. 
In the first stage we reveal a suitable subset of the forward ball of radius $t-h-1$ around node $i$ and construct the tree $\cT_i$ entering the definition of nice paths. In the second stage, conditionally on the first stage we reveal the backward ball of radius $h$ around node $j$. Finally, in the completion stage,  we sample all the remaining pairings conditionally on the first two stages. The   proof of the first part of Proposition \ref{pr:upper} is based on a concentration inequality for the random variables involved in the last stage.  

We start with the first stage. The following procedure generates the environment $\omega$ locally around a given origin $i\in V$ and extracts from it a certain subtree $\cT=\cT_i(\omega)$.  Throughout the process, we let $\partial_+\cT$ (resp. $\partial_-\cT$) denote the set of unmatched tails (resp. heads) whose vertex belongs to $\cT$.  Initially, all tails and heads are unmatched and $\cT$ is reduced to its root, $i$. We then iterate the following steps:
\begin{enumerate}
\item[1.] A tail $e\in\partial_+\cT$ is selected according to some rule, see below.
\item[2.] $e$ is matched to a uniformly chosen unmatched head $f$, forming an arc $ef$.
\item[3.] If $f$ was not in $\partial_-\cT$, then the arc $ef$ and its end-vertex are added to the subtree $\cT$. 
 \end{enumerate}
 The condition in step 3 ensures that $\cT$ remains a directed tree: any $e\in\partial_+\cT$ is accessible from the root by a unique directed path, and we define the \emph{height} $\h(e)$ as the number of vertices along that path (including the vertex of $e$), and the \emph{weight} $\w(e)$ as the inverse product of vertex out-degrees along that path (including the vertex of $e$). Our rule for step 1 consists in selecting a tail with maximal weight\footnote{using an arbitrary deterministic ordering of the set of all tails to break ties.} among all $e\in\partial_+\cT$ with height $\h(e)<t-h$ and weight $\w(e)>\wm$, where
\begin{eqnarray*}
\wm & := & \frac{\ln n}n.
\end{eqnarray*}
The procedure stops when there is no such tail, which occurs after a \emph{random number $\kappa$ of pairings}. 
Let $e_1,\dots,e_\kappa$ denote the sequence of tails selected during the  procedure, and let $\cT^1\subset \cdots\subset \cT^k = \cT$ denote the sequence of trees produced during the procedure. 
The only role of the parameter $\wm$ is to control $\kappa$. Specifically, the following holds. 
\begin{lemma} \label{lm:wbound}Amost-surely, for all $1\leq k\leq\kappa$,
\begin{eqnarray*}
\w(e_{k}) & \leq & \frac{2}{k+2}.
\end{eqnarray*}
In particular, since $\w(e_{\kappa})>\wm$ by our selection rule,
\begin{eqnarray}
\label{bound:kappa}
\kappa & \leq & \frac 2\wm \ = \cO \left( \frac n {\ln n} \right).
\end{eqnarray}
\end{lemma}
\begin{proof}
It is enough to fix $w>0$ and prove that whenever some $1\leq k\leq \kappa$ satisfies
\begin{eqnarray}
\label{lcc}
\w(e_k) & \geq &  w,
\end{eqnarray}
we necessarily also have 
 \begin{eqnarray}
 \label{ccl}
 \frac{2}{k+2} & \geq &  w.
 \end{eqnarray}
 Let us assume that (\ref{lcc}) holds. Since the condition (\ref{ccl}) is monotone in $k$ we may further assume that $k$ is maximal with property (\ref{lcc}). Thus, $\w(e_1)\geq\ldots\geq \w(e_k)\geq w>\w(e_{k+1})\geq\cdots\geq\w(e_\kappa)$. The tree $\cT^k$ may not contain all the $k$ arcs $e_1f_1,\ldots,e_kf_k$, because some of them may have caused a collision upon creation. Let us complete the tree by viewing the tails of those missing arcs as proper pending edges ending in a leaf node. Observe  that
 in the resulting tree every non-leaf node $j$ has exactly $d_j^+$ outgoing arcs, all having the same weight. In particular, this new tree  has the following properties:
\begin{enumerate}
\item there are exactly $k$ arcs;
\item every non-leaf node has at least two children; 
\item every arc has weight at least $w$;
\item the weights of the arcs incident to the leaves sum up to 1. 
\end{enumerate}
 If $x$ is the number of non-leaf nodes and $y$ the number of leaves in the tree, then from 1, $k = x + y -1$ and from 2, we have $k \geq 2x$. Hence $y = k - x + 1 \geq k/2 + 1$. 
 On the other hand, any tree satisfying properties 3-4 has at most $\frac 1w$ leaves. Thus, (\ref{ccl}) must hold.  
\end{proof}

Now, at the end of the above procedure, we let $\cE$ be the set of tails $e\in\partial_+\cT$ such that $\h(e)=t-h$. Note that 
the vertex of any tail $e\in\cE$ is at distance $t-h-1$ from node $i$, and that
\begin{eqnarray}
\label{bound:we}
\sum_{e\in \cE}\w(e) & \leq & 1.
\end{eqnarray}
Next, we turn to the second stage. We sequentially generate the \emph{backward ball of radius $h$} around a given destination $j\in V$ as we did for the forward ball in subsection \ref{sec:treelike}, but with all orientations reversed. 
Notice that some of the pairings producing the backward ball of radius $h$ may have already been revealed during the previous procedure. In any case this second stage  
creates an \emph{additional random number $\tau$ of arcs}, satisfying the crude bound
\begin{eqnarray}
\label{bound:tau}
\tau & \leq & \Delta+\cdots+\Delta^h \ = \ \cO (n^{1/10}).
\end{eqnarray}
We may now consider the set $\cF$ of heads $f$ from the end-point of which the shortest path to $j$ is unique and has length $h$. Writing $\w(f)$ for the weight of that path, we have by definition of $\wpi$, 
\begin{eqnarray}
\label{bound:wf}
\wpi(j) & \geq & \frac{1}{m}\sum_{f\in \cF}\w(f).
\end{eqnarray}
Finally, turning to the last stage,  we complete the generation of environment $\omega$ by matching the $m-\kappa-\tau$ remaining unmatched 
tails to the $m-\kappa-\tau$ remaining unmatched heads uniformly at random. 
Recalling Definition \ref{def:nice}, 
\begin{eqnarray*}
P_0^t(i,j) & = & \sum_{e\in\cE}\sum_{f\in \cF}\w(e)\w(f){\bf 1}_{\w(e)\w(f)\leq \frac{1}{n\ln^2n}}
{\bf 1}_{\omega(e)=f}.
\end{eqnarray*}
The key observation here is that, by construction, any arc $ef$ contributing to this sum must be formed during the completion stage.  Indeed, by definition of $\cE$, a tail $e\in\cE$ cannot be matched during the first stage, and if $e\in\cE$ is matched during the second stage then its vertex must belong to the backward ball of radius $h$ around $j$, and therefore it cannot be matched to $f\in\cF$ by definition of $\cF$. 
Conditionally on the $\sigma-$field $\cG$ generated by the $\kappa+\tau$ arcs formed during the first two stages, the random variable $P_0^t(i,j)$ may thus be regarded as the cost of a uniform random bijection, and we can exploit the following concentration result due to Chatterjee \cite{chatterjee2007stein}.
\begin{theorem}[\cite{chatterjee2007stein}, Proposition 1.1]Fix two finite sets $E,F$ with $|E|=|F|$ and $c\colon E\times F\to \dR_+$. Consider the random variable $Z:=\sum_{e\in E}c(e,\omega(e))$ where  $\omega\colon E\to F$ is a uniform random bijection. Then,
\begin{eqnarray*}
\PP\left(Z-\EE[Z]\geq \eta\right) & \leq & \exp\left(-\frac{\eta^2}{2\|c\|_{\infty}\left(2\EE[Z]+\eta\right)}\right),
\end{eqnarray*}
for every $\eta>0$, where $\|c\|_{\infty}=\max_{(e,f)\in E\times F}c(e,f)$. 
\end{theorem}
We apply this result to $Z=P_0^t(i,j)$, conditionally on $\cG$, 
with the cost function
$$
c(e,f)=\w(e)\w(f){\bf 1}_{\w(e)\w(f)\leq \frac{1}{n\ln^2n}}.
$$ 
Thus, in this case $\|c\|_{\infty}\leq \frac{1}{n\ln^2n}$. Since the event $\omega(e)=f$ conditioned on $\cG$  has probability at most $1/(m-\kappa-\tau)$, one has 
\begin{eqnarray*}
\EE[Z|\cG] & \leq & \frac{1}{m-\kappa-\tau}\left(\sum_{e\in\cE}\w(e)\right)\left(\sum_{f\in\cF}\w(f)\right). 
\end{eqnarray*}
Thanks to (\ref{bound:kappa}), (\ref{bound:we}), (\ref{bound:tau}) and  (\ref{bound:wf}), 
\begin{eqnarray*}
\EE[Z|\cG] & \leq & (1+\cO ( 1 / \log n))\,\wpi(j),
\end{eqnarray*}
the $\cO(\cdot)$ term being deterministic and independent of $i,j$. For $n$ large enough, \begin{eqnarray*}( 1 + \varepsilon/2) \EE[Z|\cG]& \leq & (1+ \veps) \wpi(j),\end{eqnarray*}  and choosing $\eta:=\frac{\varepsilon}{2}\EE[Z|\cG]+\frac \varepsilon n$ then yields
\begin{eqnarray*}
\max_{(i,j)\in V^2}\PP\left(P_0^t(i,j)\geq (1+\varepsilon)\wpi(j)+\frac{\varepsilon}{n}\right) & \leq & 
\EE\exp\left(-\frac{\eta^2 n\ln^2n}{2(2\EE[Z|\cG]+\eta)}\right) \\
& \leq & \exp\left(-\frac{\varepsilon^2\ln^2 n}{2(4+\varepsilon)}\right),
\end{eqnarray*}
where we have used $\eta \geq \varepsilon / n$ and $2 \EE[Z|\cG] \leq 4\eta / \varepsilon$.
We may thus take a union bound over all $(i,j)\in V^2$, establishing the first part of Proposition \ref{pr:upper}.
\subsection{Proof of the second part of Proposition \ref{pr:upper}}
It now remains to bound the quenched probability that the trajectory of the random walk of length $t$ starting at $i\in V_\star$ is not nice. The first requirement in definition \ref{def:nice}   fails with probability exactly $\cQ_{i,t}\left(\frac 1{n\ln^2 n}\right)$. A failure of the third requirement implies that the $(t-h)^{\textrm{th}}$ vertex is not in $V_\star$, which has probability $o_\PP(1)$ uniformly in $i\in V$ thanks to Proposition \ref{pr:treelike}. Regarding the second requirement, there are, by construction, only two ways of escaping from $\cT_i$ before time $t-h$: 
\begin{itemize}
\item Either the weight of the trajectory falls below $\wm=\frac{\ln n}{n}$ before time $t-h$:  the quenched probability of this is $1-\cQ_{i,t-h}(\frac{\ln n}{n})$, which is $o_\PP(1)$ uniformly in $i$ by Proposition \ref{pr:weight}.
\item Or the walk traverses an arc that violated the tree structure in step 3 above.  Proposition \ref{pr:loss} below will show that the quenched probability of this is $o_\PP(1)$ uniformly in $i\in V_\star$.
\end{itemize}
Consider again the construction of the subtree $\cT=\cT_i$ described above and recall the definition of $\kappa$ above \eqref{bound:kappa}. For $1\leq k\leq \kappa$, we let $e_kf_k$ denote the $k^\textrm{th}$ formed arc, $\cT^k$ the tree obtained after $k$ arcs have been formed, and we consider the process $(W_k)_{k\geq 0}$ defined by $W_0=0$ and then inductively, 
\begin{eqnarray*}
W_{k+1} & = & W_k+{\bf 1}_{k<\kappa}{\bf 1}_{f_{k+1}\in \partial_-\cT^{k}}\w(e_{k+1}).
\end{eqnarray*}
Thus, $W_\kappa$ is the total weight of all tails that violated the tree structure in step 3 above. Since $W_{h}=0$ for $i\in V_\star$, the following proposition is more than enough to complete the proof of Proposition \ref{pr:upper}.
\begin{proposition}
\label{pr:loss}
For any fixed $\varepsilon>0$, we have uniformly in $i\in V$, 
\begin{eqnarray*}
\PP\left(W_{\kappa}\geq W_{\lfloor2/\varepsilon\rfloor}+\varepsilon\right) & = & o\left(\frac 1n\right).
\end{eqnarray*}
\end{proposition}

\begin{proof}
Let $(\cG_k)_{k}$ be the natural filtration associated with the construction of $\cT$, i.e. $\cG_k=\sigma(f_1,\ldots,f_k)$. Note that $\kappa$ is a stopping time, that $e_{k+1}$ is $\cG_{k}-$measurable, and that the conditional law of $f_{k+1}$ given $\cG_{k}$ is uniform on the $m-k$ unmatched tails. Consequently,
\begin{eqnarray*}
\EE\left[W_{k+1}-W_{k}\bigg | \cG_{k}\right] & = & {\bf 1}_{k<\kappa}\w(e_{k+1})\frac{|\partial_-\cT^k|}{m-k},\\
\EE\left[(W_{k+1}-W_{k})^2\bigg | \cG_{k}\right] & = & {\bf 1}_{k<\kappa}\w(e_{k+1})^2\frac{|\partial_-\cT^k|}{m-k}.
\end{eqnarray*}
Using Lemma \ref{lm:wbound}, its corollary (\ref{bound:kappa}), and the crude bound $|\partial_-\cT^k|\leq \Delta (k+1)$, we arrive at
\begin{eqnarray}
\label{freedmanm}
\sum_{k=0}^{\kappa-1} \EE\left[W_{k+1}-W_{k}\bigg | \cG_{k}\right] & \leq & \frac{8\Delta/\wm }{m-2/\wm}=:\alpha,\\
\label{freedmanq}
\sum_{k=0}^{\kappa-1} \EE\left[(W_{k+1}-W_{k})^2\bigg | \cG_{k}\right] & \leq & \frac{4\Delta\ln(2e/\wm)}{m-2/\wm}=:\beta.\end{eqnarray}
where we have used $\sum_{k=0}^{\kappa-1}\frac 1{k+3}\leq 1+\ln\kappa$. In addition, Lemma \ref{lm:wbound} guarantees the a.-s. bound   
\begin{eqnarray}
\label{freedmanb}
0\ \leq & W_{k+1}-W_{k} & \leq \ \frac{2}{k+3}.\end{eqnarray}
In view of  (\ref{freedmanm}),(\ref{freedmanq}),(\ref{freedmanb}), the martingale version of Bennett's inequality due to Freedman \cite[Theorem 1.6]{freedman1975tail} ensures that for every $\varepsilon>\alpha,k\in\dN$, 
\begin{eqnarray*}
\PP\left(W_{\kappa}- W_{k} \geq \varepsilon\right) & \leq & \left(\frac{e(k+3)\beta}{2(\varepsilon-\alpha)}\right)^{\frac{(\varepsilon-\alpha)(k+3)}{2}}.
\end{eqnarray*}
Note that $\alpha$ and $\beta$ do not depend on $i$ and satisfy $\alpha=o(1)$ and $\beta=n^{-1+o(1)}$. Consequently, the right hand-side is $o\left(n^{-1}\right)$ as soon as $k+3> 2/\varepsilon$, and this concludes the proof.
\end{proof}

\section{The martingale approximation}
\label{sec:martingale}
In this section, we show that $\cB^-(\cI,h)$, the backward ball of radius $h$ around a uniform vertex $\cI$, can be accurately described by the first $h$ generations of a certain Galton-Watson tree $\cT_\star$, allowing us to approximate the sequence $\left(n\pi_{t}^{-}(\cI)\right)_{0\leq t\leq h}$ by a martingale defined directly on  $\cT_\star$. Specifically, let $\cT_\star$ be the infinite random tree with marks in $V$ obtained by the following branching process: 
\begin{itemize}
\item Generation $0$ consists of a single node (the root $o$), with mark uniformly distributed in $V$.
\item Inductively, every node $x$ in generation $t\in\dN$ independently gives birth to exactly $d_{i(x)}^-$ children, where $i(x)$ denotes the mark of $x$. These new nodes form generation $t+1$, and their marks are chosen independently from the out-degree distribution on $V$. 
\end{itemize}
Now, consider a node $x$ at generation $t\in\dN$, and let $(x_0,\ldots,x_t)$ be the unique path from $x$ to the root (thus, $x_0=x$ and $x_t=o$). Define the weight of $x$ as
\begin{eqnarray*}
\w(x)  :=  \frac{nd^-_{i(x_0)}}{m}\prod_{k=0}^{t-1}\frac{1}{d^+_{i(x_k)}}.
\end{eqnarray*}
In the case where $x=o$, we have $t=0$ and the empty product is interpreted  as equal to 1, that is
$$
\w(o):=\frac{n}m\,d_\cI^-,
$$ 
where $\cI$ is uniformly distributed over $V$.
Since 
$\EE[d^-_\cI]=\frac{m}n$ , we have $\EE[\w(o)]=1$. 
For $t\in\dN$, we now define $M_t$ as the total weight of the $t^{\textrm{th}}$ generation $\cT_\star^t$: 
\begin{eqnarray*}
M_t  :=  \sum_{x\in\cT_\star^t}\w(x).
\end{eqnarray*}
The law of the random variable $M_t$ provides a good approximation to the law of $n\pi_{t}^{-}(\cI)$, where $\cI$ denotes a uniformly chosen vertex, independent of the environment. 
\begin{proposition}[Branching process approximation]
\label{pr:bpa}
The total variation distance between the law of the random vector $(n\pi_{t}^{-}(\cI))_{0\leq t\leq h}$ and that of  $(M_t)_{0\leq t\leq h}$ is less than $\frac{\Delta^{2h+3}}{m}$. 
\end{proposition}
\begin{proof}
We may  generate $\cB^-(\cI,h)$ sequentially as we did for the forward ball in subsection \ref{sec:treelike}, with directions reversed. It is now tails that get uniformly chosen from the remaining unmatched ones. Building on  remark \ref{rk:coupling}, we may instead choose uniformly from all tails, matched or not, until a tail gets chosen whose end-point was already in the ball. The chance that this event occurs before the end of the procedure is less than 
$p={\Delta k^2}/{m}$, where  $k=\Delta+\cdots+\Delta^h\leq \Delta^{h+1}$ is a crude upper-bound on the total number of steps. This creates a coupling between the ball $\cB^-(\cI,h)$ and the first $h$ generations of the tree $\cT_\star$, under which they  coincide with probability more than $1-p$.  Moreover, on this event, $(M_t)_{0\leq t\leq h}$ equals $(n\pi_{t}^{-}(\cI))_{0\leq t\leq h}$ by construction. 
\end{proof}
This connection to $(n\pi_{t}^{-}(\cI))_{0\leq t\leq h}$ motivates a deeper study of the process $(M_t)_{t\geq 0}$. 

\begin{proposition}[The martingale]
\label{pr:martingale}
$(M_t)_{t\geq 0}$ is a martingale relative to the natural filtration of the branching process. The limit $M_\star=\lim_{t\to\infty}M_t$ exists almost-surely and in $L^2$ and for all $t\in\dN$,
\begin{eqnarray*}
\EE\left[(M_\star-M_t)^2\right] & = & \frac{n(\gamma-1)\varrho^t}{m(1-\varrho)},
\end{eqnarray*}
where $\varrho,\gamma$ were defined in Theorem \ref{th:uniform}.
\end{proposition}
\begin{proof}
Write  $\to$ for the child to parent relation in $\cT_\star$. By definition, 
\begin{eqnarray*}
M_{t+1}-M_t & = & \sum_{x\in\cT_\star^t}\frac{\w(x)}{d^-_{i(x)}}\sum_{y\to x}\left(\frac{d^-_{i(y)}}{d^+_{i(y)}}-1\right).
\end{eqnarray*}
Given $\cF_t$, the random vector $(i(y))_{y\to x}$ consists of $d^-_{i(x)}$ i.i.d. samples from the out-degree distribution on $V$. Now, for a variable $\cJ$ with this distribution, we have
 \begin{eqnarray}\label{eqss}
\EE\left[\frac{d^-_\cJ}{d^+_\cJ}\right] \ = \ \sum_{i=1}^n\frac{d^+_i}{m}\frac{d^-_i}{d^+_i}\ = \ 1
& \textrm{ and } & \EE\left[\left(\frac{d^-_\cJ}{d^+_\cJ}\right)^2\right] \ = \ \sum_{i=1}^n\frac{d^+_i}{m}\left(\frac{d^-_i}{d^+_i}\right)^2 \ = \ \gamma.
\end{eqnarray}
Consequently, $\EE\left[\left. M_{t+1}-M_t\right|\cF_t\right]=0$ and 
\begin{eqnarray*}
\VV\left[\left. M_{t+1}-M_t\right|\cF_t\right] & = & (\gamma-1)\sum_{x\in\cT_\star^t}\frac{\left(\w(x)\right)^2}{d^-_{i(x)}}=:\Sigma_t.
\end{eqnarray*}
This shows that $(M_t)_{t\in\dN}$ is a (non-negative) martingale, and that its almost-sure limit $M_\star$ satisfies 
\begin{eqnarray}
\label{eq:bracket}
\EE\left[\left(M_{t}-M_\star\right)^2\right] & = & \sum_{k=t}^\infty \EE\left[\Sigma_k\right],
\end{eqnarray}
provided the right-hand side is finite. We now compute $\EE[\Sigma_t]$ for all $t\in\dN$. First note that
\begin{eqnarray*}
\Sigma_{t+1} & = & (\gamma-1)\sum_{x\in\cT_\star^t}\left(\frac{\w(x)}{d^-_{i(x)}}\right)^2\sum_{y\to x}\frac{d^-_{i(y)}}{\left(d^+_{i(y)}\right)^2}.
\end{eqnarray*}
As in \eqref{eqss}, 
\begin{eqnarray*}
\EE\left[\frac{d^-_{\cJ}}{\left(d^+_{\cJ}\right)^2}\right] & = & \sum_{i=1}^n\frac{d^+_i}m\frac{d^-_i}{(d^+_i)^2} \ = \  \varrho.
\end{eqnarray*}
Consequently, $\EE\left[\left. \Sigma_{t+1}\right|\cF_t\right]=\varrho \Sigma_t$. In particular, $\EE[\Sigma_t]=\EE[\Sigma_0]\varrho^t$ for all $t\in\dN$. But $\Sigma_0=\frac{n^2(\gamma-1) d_i^-}{m^2}$, where $i$ denotes the mark of the root. As the latter is uniformly distributed on $V$,  we deduce that $\EE[\Sigma_0] = \frac{n(\gamma-1)}{m}$, and inserting this into (\ref{eq:bracket}) completes the computation of $\EE[(M_\star-M_t)^2]$.
\end{proof}

Now that $M_\star$ is constructed, we may establish the representation announced in the introduction.

\begin{lemma}\label{repmstar}
The random variable $M_\star$ admits the representation (\ref{eq:mstar})-(\ref{rde}).  
\end{lemma}
\begin{proof}
Let $\cT_\star^t(x)$ denote the set of all nodes $y$ from which there is a child-to-parent path of length $t$ to the node $x$ in $\cT_\star$. Writing $x_0,\ldots,x_t$ for the corresponding path, we define 
\begin{eqnarray*}
\w(y,x) & := &d^-_{i(x_0)}\prod_{k=0}^{t}\frac{1}{d^+_{i(x_k)}}.
\end{eqnarray*}
Now, for every $t\in\dN$ and every node $x$ in the tree $\cT_\star$, we define the quantity
\begin{eqnarray*}
Z_t(x) & := & \sum_{y\in\cT_\star^t(x)}\w(y, x).
\end{eqnarray*}
A martingale argument similar to the one above shows that the limit $Z_\star(x)=\lim_{t\to\infty}Z_t(x)$ exists almost-surely. By construction, we have for all $t\geq 0$, 
\begin{eqnarray*}
M_{t+1} \ = \ \frac{n}{m}\sum_{x\to o} Z_{t}(x) & \textrm{ and } & Z_{t+1}(x) \ = \ \frac{1}{d^+_{i(x)}} \sum_{y\to x}Z_{t}(y).
\end{eqnarray*}
Passing to the limit, we deduce that almost-surely
\begin{eqnarray*}
M_{\star} \ = \ \frac{n}{m}\sum_{x\to o} Z_{\star}(x) & \textrm{ and } & Z_{\star}(x) \ = \ \frac{1}{d^+_{i(x)}} \sum_{y\to x}Z_{\star}(y).
\end{eqnarray*}
The first equation is precisely (\ref{eq:mstar}), since $i(o)$ is uniformly distributed on $V$ and  conditionally on $\cF_0$, the random vector $(Z_{\star}(x),x\to o)$ consists of $d^-_{i(o)}$ i.i.d. coordinates. The second equation implies (\ref{rde}) since conditionally on $\cF_0$, each $x\sim o$ satisfies the following: $i(x)$ has the out-degree distribution and conditionnally on  $\cF_1$, the random vector $(Z_{\star}(y),y\to x)$ consists of $d^-_{i(x)}$ i.i.d. coordinates whose marginal distribution is the same as $Z_\star(x)$.
\end{proof}

\section{Proof of Theorem \ref{th:uniform}}
\label{sec:uniform}
Thanks to (\ref{eq:proxy}), the proof of Theorem \ref{th:uniform} boils down to establishing that for any fixed $t\in\dN$, \begin{eqnarray*}
2\|{\pi_{h}^{-}}-\pi_{t}^{-}\|_{\textsc{tv}} & \leq & \sqrt{\frac{n\left(\gamma-1\right) \varrho^t}{m(1-\varrho)}}+o_\PP(1).
\end{eqnarray*}
We first prove that the left-hand side is concentrated around its expectation.
\begin{lemma}
\label{lm:concentration}
For any fixed $t\in\dN$, 
\begin{eqnarray*}
\|{\pi_{h}^{-}}-\pi_{t}^{-}\|_{\textsc{tv}} & = &\EE\left[\|{\pi_{h}^{-}}-\pi_{t}^{-}\|_{\textsc{tv}}\right]+o_\PP(1).
\end{eqnarray*}
\end{lemma}
\begin{proof}
Fix the environment. The probability that the random walk of length $t$ starting from the in-degree distribution traverses a particular tail is trivially upper-bounded by the $\pi_{0}^{-}-$measure of the backward ball of radius $t$ around its vertex, which is at most $\frac{\Delta^{t+2}}{m}$. Consequently, swapping two coordinates of the environment $\omega$ cannot alter the distribution $\pi_{t}^{-}$ by more than $\frac{2\Delta^{t+2}}{m}$ in total variation. By the triangle inequality, we conclude that it cannot alter the quantity $Z=\|\pi_{t}^{-}-{\pi_{h}^{-}}\|_{\textsc{tv}}$ by more than $b=\frac{2\Delta^{t+2}+2\Delta^{h+2}}{m}$. By a now classical application of the Azuma--Hoeffding inequality (see, e.g., \cite[Section 3.2.]{MR1678578}), this property implies that under the uniform measure, the random variable $\omega\mapsto Z(\omega)$ satisfies the following concentration inequality:
\begin{eqnarray}
\label{eq:azuma}
\PP\left(|Z- \EE[Z]|\geq\varepsilon\right) & \leq & 2\exp\left(-\frac{\varepsilon^2}{2mb^2}\right).
\end{eqnarray}
We may now let $n\to\infty$. With $t$ fixed and $h$ as in (\ref{def:h}), we have $mb^2\to 0$, as desired.
\end{proof}
It only remains to bound the expectation. First, we may rewrite it as
\begin{eqnarray*}
2\EE\left[\|\pi_{t}^{-}-{\pi_{h}^{-}}\|_{\textsc{tv}}\right] & = &  \EE\left[\left|n\pi_{t}^{-}(\cI)-n{\pi_{h}^{-}}(\cI)\right|\right],
\end{eqnarray*}
where $\cI$ is a uniformly chosen vertex, independent of the environment, 
and the expectation extends over both $\omega$ and $\cI$. Now, by Proposition \ref{pr:bpa}, the total variation distance between the law of the random vector $(n\pi_{t}^{-}(\cI))_{0\leq t\leq h}$ and that of  $(M_t)_{0\leq t\leq h}$ is less than $p=\frac{\Delta^{2h+3}}{m}$. Since all coordinates  are crudely bounded by $\Delta^{h+1}$, we deduce  that
\begin{eqnarray*}
\left|\EE\left[\left|n\pi_{t}^{-}(\cI)-n{\pi_{h}^{-}}(\cI)\right|\right] - \EE\left[\left|M_t-M_h\right|\right]\right|& \leq & p\Delta^{h+1}\ = \ o(1).
\end{eqnarray*}
Finally, the Cauchy--Schwarz inequality and the orthogonality of martingale increments yield
\begin{eqnarray*}
\EE\left[|M_t-M_h|\right]^2 & \leq & \EE\left[\left(M_t-M_h\right)^2\right] \\
& \leq & \EE\left[\left(M_t-M_\star\right)^2\right],
\end{eqnarray*}
and Proposition \ref{pr:martingale} concludes the proof.

\section{Proof of Theorem \ref{th:measure}}
\label{sec:measure}
We first establish the following (weaker) result.
\begin{lemma}Let $f\colon\dR\to\dR$ be non-expansive, i.e. $|f(x)-f(y)|\leq |x-y|$ for all $x,y\in\dR$. Then,
\begin{eqnarray*}
\left|\frac 1n\sum_{i=1}^nf\left(n\pi_\star(i)\right)-\EE[f(M_\star)]\right| & \xrightarrow[n\to\infty]{\PP} & 0.
\end{eqnarray*}
\end{lemma}
\begin{proof}
The assumption on $f$ ensures that for any probability measures $\pi,\pi'$ on $V$,
\begin{eqnarray}
\label{eq:nonexp}
\left|\frac 1n\sum_{i=1}^nf\left(n\pi(i)\right)-\frac 1n\sum_{i=1}^nf\left(n\pi'(i)\right)\right| & \leq & 2\|\pi-\pi'\|_{\textsc{tv}}.
\end{eqnarray}
Choosing $\pi=\pi_\star$, $\pi'={\pi_{h}^{-}}$ and invoking (\ref{eq:proxy}), we see that 
\begin{eqnarray*}
 \frac 1n\sum_{i=1}^nf\left(n\pi_\star(i)\right) & = &  \frac 1n\sum_{i=1}^nf\left(n{\pi_{h}^{-}}(i)\right)+o_\PP(1).
\end{eqnarray*}
Now, recall from the proof of Lemma \ref{lm:concentration} that swapping two coordinates of the environment cannot alter ${\pi_{h}^{-}}$ by more than $\frac{2\Delta^{h+2}}{m}$ in total variation. In view of (\ref{eq:nonexp}), we deduce that a swap cannot alter the variable $Z=\frac 1n\sum_{i=1}^nf\left(n{\pi_{h}^{-}}(i)\right)$ by more than $b=\frac{4\Delta^{h+2}}{m}$. Since $mb^2\to 0$, the concentration inequality (\ref{eq:azuma}) implies
\begin{eqnarray*}
\frac 1n\sum_{i=1}^nf\left(n{\pi_{h}^{-}}(i)\right) & = & \EE\left[\frac 1n\sum_{i=1}^nf\left(n{\pi_{h}^{-}}(i)\right)\right]+o_\PP(1).
\end{eqnarray*}
Observe that the expectation on the right-hand side is $\EE\left[f\left(n{\pi_{h}^{-}}(\cI)\right)\right]$, where $\cI$ is a uniform  vertex, independent of the environment. Proposition \ref{pr:bpa} provides us with a coupling under which the random variables $n{\pi_{h}^{-}}(\cI)$ and $M_h$ differ with probability less than $p=\frac{\Delta^{2h+3}}{m}$. Since both are bounded by $\Delta^{h+1}$ and since $f$ is non-expansive, we obtain
\begin{eqnarray*}
\left|\EE\left[\frac 1n\sum_{i=1}^nf\left(n{\pi_{h}^{-}}(i)\right)\right]-\EE\left[f\left(M_h\right)\right]\right| 
& \leq & p{\Delta^{h+1}}\ = \ o(1).
\end{eqnarray*}
Finally, by the non-expansiveness of $f$, the Cauchy--Schwarz inequality, and Proposition \ref{pr:martingale},
\begin{eqnarray*}
\left|\EE\left[f\left(M_h\right)\right]-\EE\left[f\left(M_\star\right)\right]\right| 
& \leq & \EE\left[\left|M_h-M_\star\right|\right]\\
& \leq & \sqrt{\EE\left[\left(M_h-M_\star\right)^2\right]}\\
& \leq & \sqrt{\frac{n(\gamma-1)\varrho^h}{m(1-\varrho)}}\ = \  o(1).
\end{eqnarray*}
Combining those four inequalities concludes  the proof.
\end{proof}
To conclude the proof of Theorem \ref{th:measure}, 
we observe that $M_\star$ has the desired representation from Lemma \ref{repmstar}. Moreover, $M_\star$ is bounded in $L^2$ uniformly in $n$ by Proposition \ref{pr:martingale}, and hence uniformly integrable as $n$ varies. We may then apply the following general result. Write $\cP_1(\dR)$ for the (Polish) space of Borel probability measures on $\dR$ with finite first absolute moment, equipped with the Wasserstein distance $\cW$. 
\begin{lemma}[Convergence of random measures in $\cP_1(\dR)$]Let $(\cL_n)_{n\geq 1}$ be a sequence of random elements of  $\cP_1(\dR)$, and let $(\mathbb L_n)_{n\geq 1}$ be a uniformly integrable sequence of deterministic elements of   $\cP_1(\dR)$. Assume that  for every non-expansive function $f\colon\dR\to\dR$, 
\begin{eqnarray}
\label{start}
\left|\int_\dR f{\D} \cL_n-\int_\dR f{\D} {\mathbb L}_n\right| & \xrightarrow[n\to\infty]{\PP} & 0.
\end{eqnarray}
Then  $\cW\left(\cL_n, {\mathbb L}_n\right)  \xrightarrow[n\to\infty]{\PP} 0$.
\end{lemma}
\begin{proof}
It is classical that convergence in $\cP_1(\dR)$ can be tested on a certain countable family $\cF$ of non-expansive functions, in the following sense:
\begin{eqnarray}
\label{countable}
\cW(\dL_n,\dL)\to 0 & \Longleftrightarrow & \left(\forall f\in\cF,\ \int_\dR f{\D}\dL_n\to\int_{\dR} f{\D}\dL\right).
\end{eqnarray}
One may take, for example, the function $x\mapsto |x|$ together with the functions $x\mapsto \frac{1-e^{-\theta [x]_+}}\theta$ and $x\mapsto \frac{1-e^{-\theta [x]_-}}\theta$  for  positive rational numbers $\theta$, since convergence in $\cP_1(\dR)$ is equivalent to weak convergence together with convergence of the first absolute moment. We also recall the following standard result: a sequence of real-valued random variables $(X_n)_{n\geq 1}$ converges in probability to $0$ if and only if from every subsequence $(X_{a(n)})_{n\geq 1}$, one can further extract a subsubsequence $(X_{a(b(n))})_{n\geq 1}$ converging to $0$ almost-surely. Let us apply this to $X_n=\cW(\cL_n,\dL_n)$: given an extractor $a$ (i.e., an increasing function from $\dN$ to $\dN$), we will construct an extractor $b$ such that 
 \begin{eqnarray}
 \label{aim}
 \cW(\cL_{a(b(n))},\dL_{a(b(n))}) & \xrightarrow[n\to\infty]{\rm{a.-s.}} & 0.
 \end{eqnarray}
First recall that uniform integrability means relative compactness in $\cP_1(\dR)$: thus, the assumption ensures that there exists  $\dL\in\cP_1(\dR)$ and an extractor $c$ such that 
 \begin{eqnarray}
 \label{compactness}
\cW(\dL_{a(c(n))},\dL) & \xrightarrow[n\to\infty]{} & 0
 \end{eqnarray}
Combining this with  assumption (\ref {start}), we see that for each non-expansive function  $f\colon\dR\to\dR$,
\begin{eqnarray*}
\left|\int_\dR f{\D} \cL_{a(c(n))}-\int_\dR f{\D} {\mathbb L}\right| & \xrightarrow[n\to\infty]{\PP} & 0.
\end{eqnarray*}
 In particular, for a fixed $f$, there is an extractor $d$ such that 
\begin{eqnarray}
\label{extractor}
\left|\int_\dR f{\D} \cL_{a(c(d(n)))}-\int_\dR f{\D} {\mathbb L}\right| & \xrightarrow[n\to\infty]{\rm{a.-s.}} & 0.
\end{eqnarray}
The extractor $d$ depends on $f$, but by diagonal extraction one may construct one that satisfies (\ref{extractor}) simultaneously for all $f$ in the countable family $\cF$.  In view of  (\ref{countable}), we conclude that
\begin{eqnarray*}
\cW\left(\cL_{a(c(d(n)))}, {\mathbb L}\right) & \xrightarrow[n\to\infty]{\rm{a.-s.}} & 0.
\end{eqnarray*}
Recalling the  convergence (\ref{compactness}), we see that the extractor $b(n):=c(d(n))$ satisfies (\ref{aim}). 
\end{proof}
 \bibliographystyle{abbrv}
\bibliography{rwrd}

\begin{thebibliography}{10}

\bibitem{2015arXivRout}
L.~Addario-Berry, B.~Balle, and G.~Perarnau.
\newblock {Diameter and stationary distribution of random $r$-out digraphs}.
\newblock {\em ArXiv e-prints}, 2015.

\bibitem{aldous1983mixing}
D.~Aldous.
\newblock Random walks on finite groups and rapidly mixing {M}arkov chains.
\newblock In {\em Seminar on probability, {XVII}}, volume 986 of {\em Lecture
  Notes in Math.}, pages 243--297. Springer, Berlin, 1983.

\bibitem{aldous1986shuffling}
D.~Aldous and P.~Diaconis.
\newblock {Shuffling cards and stopping times}.
\newblock {\em American Mathematical Monthly}, pages 333--348, 1986.

\bibitem{MR1717530}
J.~Barral.
\newblock Moments, continuit\'e, et analyse multifractale des martingales de
  {M}andelbrot.
\newblock {\em Probab. Theory Related Fields}, 113(4):535--569, 1999.

\bibitem{MR3275997}
J.~Barral.
\newblock Mandelbrot cascades and related topics.
\newblock In {\em Geometry and analysis of fractals}, volume~88 of {\em
  Springer Proc. Math. Stat.}, pages 1--45. Springer, Heidelberg, 2014.

\bibitem{2015arXivNBRW}
A.~Ben-Hamou and J.~Salez.
\newblock {Cutoff for non-backtracking random walks on sparse random graphs}.
\newblock {\em ArXiv e-prints}, 2015.

\bibitem{MR3252922}
I.~Benjamini, G.~Kozma, and N.~Wormald.
\newblock The mixing time of the giant component of a random graph.
\newblock {\em Random Structures Algorithms}, 45(3):383--407, 2014.

\bibitem{2015arXiv}
N.~{Berestycki}, E.~{Lubetzky}, Y.~{Peres}, and A.~{Sly}.
\newblock {Random walks on the random graph}.
\newblock {\em ArXiv e-prints}, 2015.

\bibitem{chatterjee2007stein}
S.~Chatterjee.
\newblock Stein's method for concentration inequalities.
\newblock {\em Probability theory and related fields}, 138(1):305--321, 2007.

\bibitem{chen2008cutoff}
G.-Y. Chen and L.~Saloff-Coste.
\newblock The cutoff phenomenon for ergodic {M}arkov processes.
\newblock {\em Electronic Journal of Probability}, 13(3):26--78, 2008.

\bibitem{2014arXiv14097443}
N.~{Chen}, N.~{Litvak}, and M.~{Olvera-Cravioto}.
\newblock {Ranking algorithms on directed configuration networks}.
\newblock {\em ArXiv e-prints}, Sept. 2014.

\bibitem{cooper2011random}
C.~Cooper.
\newblock Random walks, interacting particles, dynamic networks: Randomness can
  be helpful.
\newblock In {\em Structural Information and Communication Complexity}, pages
  1--14. 2011.

\bibitem{MR2056402}
C.~Cooper and A.~Frieze.
\newblock The size of the largest strongly connected component of a random
  digraph with a given degree sequence.
\newblock {\em Combin. Probab. Comput.}, 13(3):319--337, 2004.

\bibitem{MR2157821}
C.~Cooper and A.~Frieze.
\newblock The cover time of random regular graphs.
\newblock {\em SIAM J. Discrete Math.}, 18(4):728--740, 2005.

\bibitem{MR2283218}
C.~Cooper and A.~Frieze.
\newblock The cover time of sparse random graphs.
\newblock {\em Random Structures Algorithms}, 30(1-2):1--16, 2007.

\bibitem{MR2290325}
C.~Cooper and A.~Frieze.
\newblock The cover time of the preferential attachment graph.
\newblock {\em J. Combin. Theory Ser. B}, 97(2):269--290, 2007.

\bibitem{MR2422388}
C.~Cooper and A.~Frieze.
\newblock The cover time of the giant component of a random graph.
\newblock {\em Random Structures Algorithms}, 32(4):401--439, 2008.

\bibitem{MR2885424}
C.~Cooper and A.~Frieze.
\newblock Stationary distribution and cover time of random walks on random
  digraphs.
\newblock {\em J. Combin. Theory Ser. B}, 102(2):329--362, 2012.

\bibitem{cooper2014vacant}
C.~Cooper and A.~Frieze.
\newblock Vacant sets and vacant nets: Component structures induced by a random
  walk.
\newblock {\em arXiv preprint arXiv:1404.4403}, 2014.

\bibitem{diaconis1996cutoff}
P.~Diaconis.
\newblock The cutoff phenomenon in finite {M}arkov chains.
\newblock {\em Proc. Nat. Acad. Sci. U.S.A.}, 93(4):1659--1664, 1996.

\bibitem{MR1068491}
P.~Diaconis, R.~L. Graham, and J.~A. Morrison.
\newblock Asymptotic analysis of a random walk on a hypercube with many
  dimensions.
\newblock {\em Random Structures Algorithms}, 1(1):51--72, 1990.

\bibitem{diaconis1981generating}
P.~Diaconis and M.~Shahshahani.
\newblock {Generating a random permutation with random transpositions}.
\newblock {\em Probability Theory and Related Fields}, 57(2):159--179, 1981.

\bibitem{MR2962084}
J.~Ding, E.~Lubetzky, and Y.~Peres.
\newblock Mixing time of near-critical random graphs.
\newblock {\em Ann. Probab.}, 40(3):979--1008, 2012.

\bibitem{MR2428978}
N.~Fountoulakis and B.~A. Reed.
\newblock The evolution of the mixing rate of a simple random walk on the giant
  component of a random graph.
\newblock {\em Random Structures Algorithms}, 33(1):68--86, 2008.

\bibitem{freedman1975tail}
D.~A. Freedman.
\newblock On tail probabilities for martingales.
\newblock {\em The Annals of Probability}, pages 100--118, 1975.

\bibitem{2015arXiv150200952L}
H.~{Lacoin}.
\newblock {The Cutoff profile for the Simple-Exclusion process on the circle}.
\newblock {\em ArXiv e-prints}, Feb. 2015.

\bibitem{levin2009markov}
D.~A. Levin, Y.~Peres, and E.~L. Wilmer.
\newblock {\em Markov chains and mixing times}.
\newblock American Mathematical Soc., 2009.

\bibitem{MR1439972}
Q.~Liu.
\newblock The growth of an entire characteristic function and the tail
  probabilities of the limit of a tree martingale.
\newblock In {\em Trees ({V}ersailles, 1995)}, volume~40 of {\em Progr.
  Probab.}, pages 51--80. Birkh\"auser, Basel, 1996.

\bibitem{MR1450934}
Q.~Liu.
\newblock Sur une \'equation fonctionnelle et ses applications: une extension
  du th\'eor\`eme de {K}esten-{S}tigum concernant des processus de branchement.
\newblock {\em Adv. in Appl. Probab.}, 29(2):353--373, 1997.

\bibitem{MR1741808}
Q.~Liu.
\newblock On generalized multiplicative cascades.
\newblock {\em Stochastic Process. Appl.}, 86(2):263--286, 2000.

\bibitem{MR1847093}
Q.~Liu.
\newblock Asymptotic properties and absolute continuity of laws stable by
  random weighted mean.
\newblock {\em Stochastic Process. Appl.}, 95(1):83--107, 2001.

\bibitem{2015arXivRamanujan}
E.~Lubetzky and Y.~Peres.
\newblock {Cutoff on all Ramanujan graphs}.
\newblock {\em ArXiv e-prints}, 2015.

\bibitem{lubetzky2010cutoff}
E.~Lubetzky and A.~Sly.
\newblock {Cutoff phenomena for random walks on random regular graphs}.
\newblock {\em Duke Mathematical Journal}, 153(3):475--510, 2010.

\bibitem{MR1678578}
C.~McDiarmid.
\newblock Concentration.
\newblock In {\em Probabilistic methods for algorithmic discrete mathematics},
  volume~16 of {\em Algorithms Combin.}, pages 195--248. Springer, Berlin,
  1998.

\bibitem{MR2435849}
A.~Nachmias and Y.~Peres.
\newblock Critical random graphs: diameter and mixing time.
\newblock {\em Ann. Probab.}, 36(4):1267--1286, 2008.

\bibitem{peresamerican}
Y.~Peres.
\newblock American institute of mathematics ({AIM}) research workshop ``sharp
  thresholds for mixing times'' ({P}alo {A}lto, {D}ecember 2004).
\newblock {\em Summary available at {http://www.aimath.org/WWN/mixingtimes}}.

\bibitem{MR1176497}
U.~R{\"o}sler.
\newblock A fixed point theorem for distributions.
\newblock {\em Stochastic Process. Appl.}, 42(2):195--214, 1992.

\bibitem{saloff2004random}
L.~Saloff-Coste.
\newblock {Random walks on finite groups}.
\newblock In {\em {Probability on discrete structures}}, pages 263--346.
  Springer, 2004.

\bibitem{villani}
C.~Villani.
\newblock {\em Optimal transport}, volume 338 of {\em Grundlehren der
  Mathematischen Wissenschaften [Fundamental Principles of Mathematical
  Sciences]}.
\newblock Springer-Verlag, Berlin, 2009.
\newblock Old and new.

\end{thebibliography}

\end{document}